\documentclass[12pt]{amsart}
\usepackage[applemac]{inputenc}
\usepackage{
 ulem,
 amsfonts}
\usepackage{amsmath,esint, amssymb, amsthm}

\usepackage[colorlinks=true, pdfstartview=FitV, linkcolor=blue, citecolor=blue,
 urlcolor=blue]{hyperref}
\usepackage{color}

\def\N{{\mathbb N}}
\date{}
\definecolor{sah}{rgb}{0.66,0.33, 0.04}
\definecolor{adel4}{cmyk}{1,0,0,0}
\definecolor{adel3}{rgb}{0.66,0.33, 0.04}
\definecolor{adel1}{cmyk}{0,0.20,1,0}
\definecolor{adel2}{cmyk}{0,0.40,1,0.30}
\definecolor{adel0}{rgb}{0.99,0.60, 0.30}
\definecolor{trut}{rgb}{0.99,0.80, 0.00}
\definecolor{trus}{rgb}{0.00, 0.50, 0.00}
 \definecolor{trust}{rgb}{0.99, 0.99, 0.80}
\definecolor{MaCouleur}{rgb}{0,0.9,0.3}
\def\virgp{\raise 2pt\hbox{,}}
\def\({\left(}
\def\){\right)}

\def\<{\langle}
\def\>{\rangle}

\theoremstyle{plain}

\newtheorem{Theo}{Theorem}
\newtheorem{Lemm}{Lemma}

 \newtheorem{prop}{Proposition}
 
 \newtheorem{Defin}{ Definition}

 \theoremstyle{definition}
 \newtheorem*{gracies}{Acknowledgements}
 \newtheorem*{remas}{Remarks}
 \newtheorem*{rema}{Remark}

\newcommand{\CC}{\mathbb{C}}

\newcommand{\R}{{\mathbb R}}
\newcommand{\RR}{{\mathbb R}}

\usepackage[textmath,displaymath,floats,graphics,delayed]{preview}
 \title[]{On the rotating doubly connected patches}


 \title[]{On rotating doubly connected vortices }

\author[T. Hmidi]{Taoufik Hmidi}
\address{IRMAR, Universit\'e de Rennes 1\\ Campus de
Beaulieu\\ 35~042 Rennes cedex\\ France}
\email{thmidi@univ-rennes1.fr}
\author[J. Mateu]{Joan Mateu}
\address{Departament de Matem\`{a}tiques\\
Universitat Aut\`{o}noma de Barcelona\\
08193 Bellaterra, Barcelona, Catalonia} \email{ mateu@mat.uab.cat}
\author[J. Verdera]{Joan Verdera}
\address{Departament de Matem\`{a}tiques\\
Universitat Aut\`{o}noma de Barcelona\\
08193 Bellaterra, Barcelona, Catalonia} \email{ jvm@mat.uab.cat}

\begin{document}

\begin{abstract}
In this paper we consider rotating doubly connected vortex patches
for the Euler equations in the plane. When the inner interface is an
ellipse we show that the exterior interface must be a confocal
ellipse. We then discuss some relations, first found by Flierl and
Polvani, between the parameters of the ellipses, the velocity of
rotation and the magnitude  of the vorticity in the domain enclosed by
the inner ellipse.
\end{abstract}

\maketitle

\begin{quote}
\footnotesize\tableofcontents
\end{quote}

\section{Introduction}
The motion of an  incompressible ideal fluid with constant density
is described by the Euler equations
\begin{equation*}
\label{E}
 \left\{
\begin{array}{ll}
\partial_t u(t,x)+u(t,x)\cdot\nabla u(t,x)+\nabla P(t,x)=0,\qquad x\in \mathbb R^d, t>0, \\
\textnormal{div }u(t,x)=0,\\
u(0,x)= u_{0}(x),
\end{array} \right.
     \end{equation*}
 where    \mbox{$u: \mathbb [0,T]\times\mathbb R^d\to \RR^d$}, $T>0$,
  denotes  the velocity field  of the fluid particles and  the scalar function
\mbox{$P$} stands for  the  pressure. The incompressibility
condition is an immediate consequence of the continuity equation
when the  density is assumed to be constant.  The mathematical
theory for this classical system is very reach. Many results were
devoted to the  local well-posedness problem in different function
spaces. Roughly speaking, it is well-known that the Cauchy problem
is locally well-posed if the initial velocity is above the scaling
of $C^1$ functions, for example, if $u_0\in H^s, s>\frac d2+1$. For
more details about this we refer the reader to \cite{BM,Ch, Kato}
and the references therein. Whether these solutions develop
singularities in finite time is still an open problem, apparently
very difficult. In dimension $2$ global regularity was proved long
ago \cite{Wolib} and
 extensions to the context of axisymmetric flows without swirl were
 obtained in \cite{Uk}.  In these cases the global existence follows from a special
structure of the vorticity, which yields some strong conservation
laws.

In this paper we shall focus on the vorticity dynamics  in the
plane.
  In this case the vorticity can be identified with the scalar function  $\omega=\partial_1 u^2-\partial_2 u^1$
  and its evolution is governed by the transport equation
 \begin{equation}\label{vorts1}
\partial_t\omega+u\cdot\nabla \omega=0,
\end{equation}
which amounts to saying that vorticity is conserved along particle
trajectories. This yields the  conservation laws
$$
\|\omega(t)\|_{L^p}=\|\omega_0\|_{L^p}, \quad p\in [1,\infty].
$$
The preservation of vorticity along trajectories allows to go beyond
the limitations inherent to the general theory of the hyperbolic
systems  and to show existence and uniqueness of a global weak
solution under the assumption that $\omega_0\in L^1\cap L^\infty$.
This remarkable result was proved by Yudovich in \cite{Y1}.
Uniqueness follows from the fact that the gradient of the velocity
belongs to all $L^p$ spaces and its $L^p$ norm obeys the slow growth
condition
$$
\sup_{p\geq2}\ \frac{\|\nabla v(t)\|_{L^p}}{p}<\infty.
$$
This framework offers  new perspectives for study  and allows for
example to deal rigorously with vortex patch structures in which the
vorticity takes finitely many values over a bounded region. More
precisely, we say that the initial vorticity is a patch if it takes
a non-zero constant value $c_0$ on a bounded domain $D$ and vanishes
elsewhere. In fact, we normalize so that $c_0 =1$ and the initial
vorticity is the characteristic function of the domain $D.$ Since
vorticity is preserved on particle trajectories, it can be recovered
by the formula
$$\omega(t)=\chi_{D_t},\quad D_t\triangleq\psi(t,D)$$
where  $\psi$ is the flow associated with the velocity field $u$,
that is, the solution of the ODE
$$\psi(t,x)=x+\int_0^tu(\tau,\psi(\tau,x))d\tau.
$$
Remark  that a vortex patch can be seen as a desingularization of a
point  vortex and provides a suitable mathematical model   to study
the effects of  finite vortex cores.

In the vortex patch problem the dynamics of the vorticity is reduced
to the motion of the one-dimensional boundary curve $\partial D_t$
according to the nonlocal equation
\begin{equation}\label{contour}
\partial_tz=\frac{1}{2\pi}\oint_{\partial D_t}\log|z-\xi|\, d\xi,
\end{equation}
which is referred to in the literature as the contour dynamics
equation. The problem of global existence of smooth solutions to the
contour dynamics equation was solved by Chemin \mbox{in \cite{Ch}}
(see \cite{BC} for a short proof). He proved that if the initial
boundary $\partial D$ belongs to the H\"older class $C^{s}, 1<s<2$,
then the boundary $\partial D_t$ remains in the same class for any
positive time. In particular, singularities like corners or cusps
cannot appear in finite time if the boundary of the initial domain
$D$ is smooth.

\quad In general, the motion of the boundary of the patch is
extremely complex, due to the nonlinear effects of the induced
velocity. There is in the literature only one explicit solution in
the simply connected case (that is, the case in which $D$ is a
simply connected domain). It is the Kirchhoff elliptical vortex, in
which $D$ is an ellipse with semi-axes $a$ and $b$. The motion of
this vortex patch is simply a rotation around the center of mass of
$D$ with angular velocity $\Omega=\frac{ab}{(a+b)^2}.$ When $a=b$
one obtains the circular steady solution known as the Rankine
vortex. See \cite{BM, Kirc, L}.

 The behavior of elliptical patches in an external field was first studied by
Chaplygin in \cite{Chapl} for  a pure shear. He proved that the
vortex retains its elliptical shape, rotates with variable angular
velocity and pulsates according to a certain law.  This result was
extended by Kida \cite{Kida} and Neu \cite{Neu} for a uniform
straining field and it was found that  the vortex   exhibits various
types of motion depending on the magnitudes of the strain. For
example, for weaker straining field the vortex can rotate or nutate.
However, for strong strain the vortex elongates indefinitely.
Comprehensive and up-to-date surveys of the analytical techniques
are provided in \cite{ New,Saf}.

\quad A vortex patch that rotates, like the Kirchhoff ellipses, is
called a rotating vortex patch or a V-state (for vortex state). This
terminology was introduced by Deem and Zabusky in \cite{DZ}, where
the contour dynamics equation \eqref{contour} was solved numerically
to show existence of V-states having an $m$-fold symmetry for any
$m=2,3,...$ ($m=2$ are the Kirchhoff ellipses). The reader is urged
to consult \cite{WOZ} where pictures of $m$-fold symmetric V-states
and their ``limiting" shapes are shown. A rigorous study including a
proof of existence of non circular $m$-fold symmetric V-states was
peformed by Burbea in \cite{B}. He used  conformal mappings combined
with a bifurcation analysis. The authors showed recently in
\cite{HMV} that close to the circle of bifurcation the V-states are
convex and  have $C^\infty$ boundaries. Global bifurcation in this
context has not been studied.

\quad The evolution of a system of   $N$ disjoint patches  is in
general very  complicated to analyze and each individual patch
varies in response to the self-induced velocity field  and to that
of other patches. Thus it seems to be very difficult to find
explicit solutions as for the single rotating  patches.  The most
common approximate model used to  track the vortex dynamics  is the
{\it moment model} of Melander, Zabusky and Styczek \cite{Mela}
leading to a self-consistent system of ordinary differential
equations governing the local geometric moments.  Its truncated
model is highly effective to treat.  For example,  the interaction
between several Kirchhoff ellipses which are far apart can be dealt
with. For a valuable discussion about this subject see \cite{New}. A
general review about vortex dynamics can be found in \cite{A}.

\quad The main goal of this paper is to study the rigid-body motion
(that is, rotation with constant angular velocity) for a linear
superposition of finitely many increasing patches. For the sake of
clarity and simplicity we only deal with the case of two bounded
simply connected domains $D_2$ and $D_1$ such that the closure of
$D_2$ is contained in $D_1.$  The initial vorticity is of the form
\begin{equation}\label{exp}
\omega_0=\chi_{D_1}+(\alpha-1)\chi_{D_2},
\end{equation}


so that the parameter $\alpha$ represents the magnitude of the
initial vorticity in the interior domain $D_2$. Clearly the initial
vorticity is $1$ on $D_1 \setminus \overline{D_2}$ and $0$ off
$D_1.$  By the conservation of vorticity along trajectories the
vorticity $\omega_t$ at time $t$ is of the form
\begin{equation}\label{exp2}
\omega_t=\chi_{D_{1t}}+(\alpha-1)\chi_{D_{2t}},
\end{equation}
for some domains $D_{1t}$  and $D_{2t}. $  We say that the solution
$\omega_t$ of \eqref{vorts1} with initial datum \eqref{exp} rotates
uniformly if $\omega_t$ is a uniform rotation of the initial
vorticity, namely
\begin{equation*}\label{exp3}
\omega_t(z-c)= \omega_0(e^{- i t \Omega} (z-c)), \quad z \in
\mathbb{C},
\end{equation*}
where
 $$c = \frac{1}{|D_1|} \int_{D_1} z\, \omega_0(z)\,dA(z)$$
is the center of mass of $\omega_0$ and $\Omega$ is some real number
that has to be found. The problem we consider consists in finding
the couples of domains $D_1$ and $D_2$ such that the multi-vortex
$\omega_t$ rotates.  Notice that for $\alpha=0$ we are considering
an initial doubly connected vortex patch $D_1 \setminus
\overline{D_2}$ and we are asking under what conditions this patch
rotates uniformly (around its center of mass). The question was
raised by Luis Vega and was the initial motivation for this work. An
annulus is the only  known  explicit solution of the doubly
connected vortex patch problem.

 When the
domains are confocal ellipses Flierl and Polvani found in
\cite{Flierl} the complete solutions to this problem by using
elliptical coordinates. In that work the authors dealt with finitely
many ellipses and special attention was devoted to the stability
condition in  the case of two confocal ellipses, thus generalizing
the known result of Love for the Kirchhoff elliptical vortex.

It seems that no other explicit solutions of the form under
consideration can be found in the literature. The aim in this paper
is to solve completely the problem of the rigid-body motion in the
particular case  when the interior interface $\partial D_2$ is an
ellipse. We shall prove that under this constraint the vortices of
Flierl and Polvani are the only solutions. Our first result concerns
Rankine vortices and reads as follows. Let $\Gamma_j = \partial D_j,
\, j=1,2.$

\begin{Theo}\label{impresuld}
Let $\omega_0$ be an initial   vorticity of the form \eqref{exp} and
assume that the solution $\omega_t$ \eqref{exp2} rotates uniformly.
If $\Gamma_1$ or $\Gamma_2$ is a circle then necessarily the other
curve must be a circle with the same center.
\end{Theo}
Accordingly,  if one of the curves is a circle and the second one is
not then there is no rotation and the dynamics of the vorticity is
not easy to track.

The second result  deals with the generalized Kirchhoff vortices.
Before stating it we need to introduce a piece of notation. For an
ellipse with semi-axes $a$ and  $b$ define
$$
Q\triangleq\frac{a-b}{a+b}\cdot
$$

\begin{Theo}\label{impresul}
Let $\omega_0$ be an initial   vorticity of  the form \eqref{exp}.
Assume that the interior curve  $\Gamma_2$ is an ellipse  and
$\Gamma_1$ is a Jordan curve of class $C^1$. Then the solution
$\omega_t$ rotates uniformly if and only if the following two
conditions are satisfied.
\begin{enumerate}
\item The curve $\Gamma_1$ is an ellipse with   the same foci as $\Gamma_2.$
\item The numbers $Q_1,Q_2,\alpha$ and the angular velocity  $\Omega$ satisfy,
$$
\Omega=\alpha\,\frac{Q_2^2-1}{4 Q_2^2},\quad
Q_1=Q_2\left(\frac{\alpha}{Q_2^2}+1-\alpha  \right)\quad
\hbox{and}\quad \frac{Q_2^2}{Q_2^2-1}<\alpha<0.
$$
\end{enumerate}
\end{Theo}
Before giving a brief account of the proofs some remarks are in
order.
\begin{remas}
\begin{enumerate}
\item For the doubly connected patches $(\alpha=0),$ if the interior curve is a non degenerate  ellipse
(different from a circle) then there is no rotation.
\item We believe that Theorem \ref{impresul} holds  when the exterior curve is an
ellipse. That is, if we assume that the exterior interface is an
ellipse, then one should conclude that the interior interface is an
ellipse too.
 This depends on an inverse problem that we have not ben able to solve.
\item The constraints on the parameters detailed in  $(2)$ of Theorem 2 coincide with the  ones given in \cite{Flierl}.
\item We can easily check from the expression of $Q_1$ that $0<Q_1<Q_2.$  This is consistent  with the fact that the ellipses are confocal and  $\Gamma_1$
 lies outside the domain $D_2$ enclosed by $\Gamma_2.$
\end{enumerate}
\end{remas}
We present now an outline of the proofs of our two results. We first
derive the equations governing the motion of the boundaries in the
general framework considered here. The uniform rotation condition is
shown to be equivalent to a system of two steady nonlocal equations
of nonlinear type coupling the Cauchy transforms of the domains
$D_1$ and $D_2$. It is hopeless to solve completely this system
because of its higher degree of complexity. Nevertheless, when the
interior boundary $\Gamma_2$ is an ellipse we obtain an explicit
formula for the Cauchy transform of the unknown domain $D_1$. This
leads to an inverse problem of the following type: one knows the
Cauchy transform of a domain and  one wants to determine the domain.
It is well-known that this is not always possible \cite{Var}. It is,
however, possible in our special situation by using Schwarz
functions and the maximum principle for harmonic functions. Once we
know that $\Gamma_1$ is an ellipse we come back to the system in
order to find the compatibility conditions which will in turn fix
all the involved parameters.

The paper is structured as follows. In section $2$ we  gather some
general facts about rotating vortices.  In section $3$ we derive the
equations of motion of the boundaries via the Cauchy transforms of
the domains $D_1$ and $D_2$. In section $4$ we review some useful
tools from complex analysis and potential theory and we discuss some
inverse problems. The last section is devoted to the proofs of the
main results.
\section{Preliminaries on rotating vortices}
In this section we discuss some elementary facts on vortex dynamics
for incompressible Euler equations. Recall that vorticity
$\omega=\partial_1 v^2-\partial_2 v^1$ satisfies the transport
equation \eqref{vorts1}.

We shall focus on the vortices whose  dynamics undergoes  a planar
rigid-body motion. In this case the motion can be described by a
combination of translations and rotations.  For the sake of
simplicity we restrict the study to the group of rotations.
\begin{Defin}\label{defa1}
Let $\omega_0\in L^1\cap L^\infty.$ We say that $\omega_0$ is a
rotating vorticity if the solution $\omega$ of equation
\eqref{vorts1} with initial condition $\omega_0$ is given by
$$
\omega(t,x)=\omega_0(\bold{R}_{x_0,-\theta(t)}x),\quad\quad x\in
\RR^2.
$$
Here we denote by $\bold{R}_{x_0,\theta(t)}$  the planar rotation of
center  $x_0$ and \mbox{angle $\theta(t)$}. Moreover we assume that
the function $t\mapsto\theta(t)$ is smooth and non-constant.
\end{Defin}
In  the vortex patch class this definition reduces to the following:
 $\omega_0=\chi_{D}$, with $D$ a bounded domain, is a rotating vorticity (or, equivalently, $D$ is a rotating
vortex patch or V-state) if and only if
$$
\omega(t)=\chi_{D_t}\quad\hbox{and}\quad D_t\triangleq
\bold{R}_{x_0,\theta(t)}D.
$$
In the preceding definition the vorticity $\omega_0$ is assumed to
be bounded and integrable in order to get a unique global solution
according to Yudovich's theorem.
The velocity  dynamics in the framework of rotating vortices is
described as follows.
\begin{prop}\label{vel-eq}
Let $\omega_0$ be a rotating vorticity as in Definition
${\ref{defa1}}$. Then the velocity $v(t)$ can be recovered from the
initial velocity $v_0$ according to the formula
$$
v(t,x)=\bold{R}_{x_0,\theta(t)} v_0(\bold{R}_{x_0,-\theta(t)}x).
$$
\end{prop}
\begin{proof}
We shall use the formula
$$
\Delta v(t,x)=\nabla^\perp\omega(t,x).
$$
Performing some algebraic computations we get
\begin{eqnarray*}
\nabla^\perp\omega(t,x)&=&\bold{R}_{x_0,\theta(t)}\nabla^\perp\omega_0(\bold{R}_{x_0,-\theta(t)}x)\\
&=& \bold{R}_{x_0,\theta(t)}\Delta v_0(\bold{R}_{x_0,-\theta(t)}x)\\
&=& \Delta\big(\bold{R}_{x_0,\theta(t)}
v_0(\bold{R}_{x_0,-\theta(t)}x)\big).
\end{eqnarray*}
Thus the velocity fields  $x\mapsto \bold{R}_{x_0,\theta(t)}
v_0(\bold{R}_{x_0,-\theta(t)}x)$  and $v$ differ by a harmonic
function and both decay at infinity. Hence they are equal.
\end{proof}

\begin{Defin}
Let $\omega$ be a compactly supported solution of \eqref{vorts1}
with non-zero total mass
$$m(t)\triangleq {\int_{\RR^2}\omega(t,x)dx }.$$
Define the center of mass of  $\omega$ as
$$
X(t)=\frac{1}{m(t)}{\int_{\RR^2}x\,\omega(t,x)\,dx}.
$$
\end{Defin}
The total mass and the center of mass are invariants of the motion.
We include a short proof of this classical fact for the sake of the
reader.
\begin{prop}\label{cons1}
Let $\omega_0$ be a smooth compactly supported initial vorticity
with non zero total mass. Then for any positive time $t$
$$
m(t)=m(0)\quad \hbox{and}\quad X(t)=X(0).
$$
\end{prop}
\begin{proof}
The conservation of mass follows easily from the characteristics
method. The vorticity $\omega(t)$ can be expressed in terms of its
initial value and the flow $\psi$ according to the formula
$$
\omega(t,x)=\omega_0(\psi^{-1}(t,x)).
$$
The incompressibility condition entails that the flow preserves
Lebesgue measure and thus the mass is conserved in time.

The invariance of the center of mass follows from the constancy of
the functions
$$
f_j(t)\triangleq \int_{\RR^2}x_j\,\omega(t,x)\,dx, \quad j=1,2.
$$
Differentiation of the functions $f_j$ with respect to the time
variable combined with the vorticity equation \eqref{vorts1} yields
\begin{eqnarray*}
f_j^\prime(t)&=&\int_{\RR^2}x_j\,\partial_t\omega(t,x)\,dx\\
&=&-\int_{\RR^2}x_j\, (v\cdot\nabla \omega)(t,x)\,dx\\
&=&\int_{\RR^2}v^j(t,x)\, \omega(t,x)\,dx.
\end{eqnarray*}
Since $\omega=\partial_1 v^2-\partial_2 v^1,$ an integration  by
parts yields
\begin{eqnarray*}
f_1^\prime(t)&=&\int_{\RR^2}v^1(t,x) \, (\partial_1 v^2(t,x)-\partial_2 v^1 (t,x))\,dx\\
&=&\int_{\RR^2}v^1(t,x)\, \partial_1 v^2(t,x)\,dx\\
&=&-\int_{\RR^2}\partial_1v^1(t,x)\,v^2(t,x)\,dx\\
&=&\int_{\RR^2}\partial_2v^2(t,x)\,v^2(t,x)\,dx\\
&=&0.
\end{eqnarray*}
\end{proof}
We now explore the relationship between center of rotation and
center of mass.
\begin{prop}\label{centremass}
Let $\omega_0=\chi_{D}$ be a vortex patch with non zero total mass,
which rotates around the point $x_0,$ . Then necessarily $x_0$ is
the center of mass of the domain $D.$
\end{prop}
\begin{proof}
By a change of variables
\begin{eqnarray*}
X(t)&=&\frac{1}{m(0)}\int_{\RR^2}x\,\omega_0(\bold{R}_{-\theta(t),x_0} x)dx\\
&=&\frac{1}{m(0)}\int_{\RR^2}(\bold{R}_{\theta(t),x_0}x)\omega_0(x)dx\\
&=&\frac{1}{m(0)}\bold{R}_{\theta(t),x_0}\left(\int_{\RR^2}x\omega_0(x)dx\right)\\
&=&\bold{R}_{\theta(t),x_0} X(0).
\end{eqnarray*}
Since $X(t)=X(0)$ by Proposition \eqref{cons1},  $X(0)$ is fixed by
the rotation and thus $X(0)=x_0$, as claimed.
\end{proof}
The last result of this section is the nontrivial fact that any
rotating patch must have a constant angular velocity.
\begin{prop}\label{pruss}
Let $\omega_0=\chi_{D}$ be a rotating vortex patch different from
the Rankine vortex.  Then the angular velocity is necessarily
constant, that is,
$$
\theta(t)= t \Omega+ \theta_0,  \quad  t\geq0,
$$
for some constants $\Omega$ and $\theta_0.$
\end{prop}
\begin{proof}
Let $s\mapsto \gamma_t(s)$ be a parametrization of the boundary of
the patch $D_t\triangleq \bold{R}_{x_0,\theta(t)}D.$ Then, as we
will prove in the next section, the motion of the boundary $\partial
D_t$ satisfies equation \eqref{boundary}, namely
$$
\hbox{Im}\big((\partial_t\gamma_t-v(t,\gamma_t)\overline{\gamma_t^\prime}\big)=0.
$$
Here the prime denotes derivative with respect to the $s$ variable.
This equation leads to \eqref{triss1}, that is,
$$
\dot{\theta}(t)\hbox{Re}\left(\gamma_0\overline{\gamma_0^\prime}\right)=\hbox{Im}\left(\,
v_0(\gamma_0)\,\overline{\gamma_0^\prime}  \right),
$$
which is equivalent to
$$
\frac{\dot{\theta}(t)}{2}\frac{d}{ds}|\gamma_0(s)|^2=\hbox{Im}\left(\,
v_0(\gamma_0)\,\overline{\gamma_0^\prime}  \right).
$$
If there exists some $s$ with $\frac{d}{ds}|\gamma_0(s)|^2\neq 0$
then, since the right-hand side does not depend on the time
variable, we conclude that $\dot{\theta}(t)$ is constant. Otherwise,
$\frac{d}{ds}|\gamma_0(s)|^2$ vanishes everywhere, which tells us
that the initial domain is a disc. Thus our vortex is the Rankine
vortex, which rotates with any angular velocity.
\end{proof}
\section{Boundary motion}
We shall in what follows describe the motion of a piecewise  constant vorticity in the plane.
  Let $D_j,  \; 1 \le j \le n, $ be a family of simply connected domains such that for each $j$
the closure of $D_{j+1}$ is contained in $D_j$.  Assume moreover
that the boundary $\Gamma_j$ of $D_j$ is a Jordan curve of class
$C^1,$ $ \, 1 \le j \le n .$  We set $E_j=D_{j}\backslash
\overline{D_{j+1}}$ for $1 \le j \le n-1$ and $E_n= D_n.$  Let
$\alpha_j, 1 \le j \le n,$ be a family of real numbers such that
$\alpha_1\neq 0$ and $\alpha_j\neq \alpha_{j+1}$ for $1 \le j \le
n-1$. Now take an initial vorticity of the form
\begin{equation}\label{multiconss}
\omega_0=\sum_{j=1}^n\alpha_j\,\chi_{E_j},
\end{equation}
where $\chi_{E_j}$ denotes the characteristic function of $E_j.$
 Since the vorticity is conserved along the particle trajectories, the  initial structure of the  vorticity
  is preserved in time. Thus the vorticity at time $t$ has the form
\begin{equation}\label{multicon11}
\omega(t)=\sum_{j=1}^n\alpha_j\,\chi_{E_{j,t}}, \quad
E_{j,t}=\psi(t,E_j),
\end{equation}
where $\psi$ is the flow map
\begin{equation}\label{flow}
\psi(t,x)=x+\int_0^tv(\tau,\psi(\tau,x))d\tau
\end{equation}
 associated with the velocity $v.$  We will describe the dynamics
of the interfaces $\Gamma_{j,t}\triangleq \psi(t, \Gamma_j), 1 \le j
\le n .$  In particular, the case $n=1$ gives the equation for the
boundary motion of a simply connected vortex patch and the case
$n=2$ and $\alpha_1 = 1, \alpha_2 =0,$ provides the system of two
equations for the boundary of a doubly connected rotating vortex
patch. There are at least two natural ways to derive the equations
of the boundary.

\subsection{First approach}
The motion of the interfaces $\Gamma_{j,t},\;  1 \le j \le n ,$ is
subject to the kinematic constraint that the boundary is transported
with the flow. In particular, it is a material surface and thus
there is no flux matter across the boundary. Since we assume that
the interfaces $\Gamma_j$ are $C^1$-smooth we can express
$\Gamma_j$, for each fixed $j$ satisfying $1 \le j \le n$, as
\begin{equation*}\label{gammaj}
\Gamma_j = \{x \in \mathbb{R}^2 : \varphi_j(x) = 0 \},
\end{equation*}
where $\varphi_j$ is a real function of class $C^1$ on the plane,
such that $\nabla \varphi_j(x) \neq 0, \; x \in \Gamma_j,$
$\varphi_j < 0$ on $D_j$ and $\varphi_j > 0$ on $ \mathbb{R}^2
\setminus \overline{D_j}.$  One says that $\varphi_j$ is a defining
function for $\Gamma_j.$ Set
\begin{equation*}\label{fijt}
F_j(t,x)= \varphi_j(\psi^{-1}(t,x)),
\end{equation*}
where $\psi$ is the flow \eqref{flow}. Then $x \rightarrow F_j(t,x)$
is a defining function for $\Gamma_{j,t} = \psi(t,\Gamma_j).$ Since
by definition $F_j(t,x)$ is transported by the flow , it satisfies
the transport equation
 $$
 \partial_t F+v\cdot\nabla F=0.
 $$
 Now, let $\gamma_t(s)$ be a parametrization of  $\Gamma_{1,t},$ continuously differentiable in $t$,
  and let $\vec{n}_t$ be the unit outward normal vector to $\Gamma_{1,t}.$ Differentiating
the equation $F(t,\gamma_t(s))=0$ with respect to $t$ yields
 $$
 \partial_t F+\partial_t\gamma_t\cdot\nabla F=0.
 $$
 Since for $x \in \Gamma_{j,t}$ the vector  $\nabla F(t,x)$ is perpendicular to  $\Gamma_{j,t}$,
 we obtain
 \begin{equation}\label{boundarys}
 (\partial_t\gamma_t-v(t,\gamma_t))\cdot\vec{n}_t=0.
 \end{equation}
 The meaning of \eqref{boundarys} is that the velocity of the boundary and the the velocity of the fluid particle occupying
 the same position   have the same normal components.
 We observe that equation \eqref{boundarys} can be written in the complex form
 \begin{equation}\label{boundary}
 \hbox{Im}  \Big\{(\partial_t\gamma_t-v(t,\gamma_t))\overline{\gamma_t^\prime}\Big\}=0
 \end{equation}
 where the "prime" denotes  derivative with respect to the $s$ variable.

 We now take a closer look at the case of a rotating doubly connected vortex patch.
 Assume that the two interfaces rotate with the same angular velocity $\dot{\theta}(t)$ around some point,
  which can be assumed to be the origin. Denote by $\gamma_0$ a parametrization of one of the initial interfaces.
   Then  $\gamma_t(s)=e^{i\theta(t)}\gamma_0(s)$ is a parametrization of the transported interface at time $t$
   and, on one hand, we get
 \begin{eqnarray*}
  \hbox{Im}  (\partial_t\gamma_t\overline{\gamma_t^\prime})&=&\dot{\theta}(t)\,\hbox{Re}(\gamma_t\overline{\gamma_t^\prime})\\
  &=&\dot{\theta}(t)\,\hbox{Re}(\gamma_0\overline{\gamma_0^\prime}).
 \end{eqnarray*}
 By Proposition \ref{vel-eq}
\begin{eqnarray*}
v(t,\gamma_t)&=&e^{i\theta(t)} v_0(e^{-i\theta(t)}\gamma_t)\\
&=&e^{i\theta(t)} v_0(\gamma_0).
\end{eqnarray*}
Hence, on the other hand,
$$
 \hbox{Im}  (v(t,\gamma_t)\overline{\gamma_t^\prime})= \hbox{Im}  \big(v_0(\gamma_0)\overline{\gamma_0^\prime}\,\big).
$$
Therefore \eqref{boundary} becomes
\begin{equation}\label{triss1}
\hbox{Im}  \big(v_0(\gamma_0)\overline{\gamma_0^\prime}\,\big) =
\hbox{Im}  (v(t,\gamma_t)\overline{\gamma_t^\prime}) =  \hbox{Im}
(\partial_t\gamma_t\overline{\gamma_t^\prime})= \dot{\theta}(t)
\hbox{Re}(\gamma_0\overline{\gamma_0^\prime}) .
\end{equation}
It follows from the identity above, as we remarked before, that the
angular velocity $\dot{\theta}(t)\equiv \Omega $ is constant.
 Recall
that  $\partial_z=\frac12(\partial_x-i\partial_y)$ and let $\Psi$
stand for the stream function at time $0$, namely,
\begin{equation*}\label{logpot0}
\Psi(z)=\frac{1}{2\pi}\int_{\RR^2}\omega_0(\xi) \log|z-\xi|\,d\xi,
\quad z \in \mathbb{C}.
\end{equation*}
Then $v_0(z) = 2 i \partial_z\Psi(\gamma_0)$ and so
\begin{equation}\label{Imv}
 \hbox{Im}  \big(v_0(\gamma_0)\overline{\gamma_0^\prime}\,\big) = 2  \hbox{Re} \big(
 \partial_z\Psi(\gamma_0) \gamma_0 ' \big).
\end{equation}
Combining \eqref{triss1} and \eqref{Imv} we conclude that the
initial interfaces
$\Gamma\triangleq\displaystyle{\cup_{j=1}^n\Gamma_{j}}$ satisfy the
system of $n$ equations
\begin{equation}\label{eq1}
2\hbox{Re}\big\{\partial_z\Psi(z)\, z^\prime \big\}= \Omega
\,\hbox{Re}\big\{ \overline{z}\,z^\prime \big\},\quad z\in \Gamma,
\end{equation}
where $z^\prime$ denotes  a tangent vector to the boundary at the
point $z$.

\subsection{Second approach}
We will give another way to derive the equation \eqref{eq1}, which
consists in analyzing directly the vorticity equation. According to
the Definition \ref{defa1} and assuming that the center of rotation
is the origin a rotating vorticity has the structure
$\omega(t,x)=\omega_0(\bold{R}_{-\theta(t)} x).$
 Straightforward computations show that
$$
\nabla\omega(t,x)=\bold{R}_{\theta(t)}(\nabla\omega_0(\bold{R}_{-\theta(t)}x)).
$$
Combining this identity with Proposition \ref{vel-eq} yields
\begin{eqnarray*}
v\cdot\nabla\omega(t,x)&=&\langle \bold{R}_{\theta(t)}
v_0(\bold{R}_{-\theta(t)} x),
\bold{R}_{\theta(t)}(\nabla\omega_0(\bold{R}_{-\theta(t)}x))\rangle
\\
&=& \langle v_0(\bold{R}_{-\theta(t)} x),  \nabla\omega_0(\bold{R}_{-\theta(t)} x)\rangle\\
&=&(v_0\cdot\nabla \omega_0)(\bold{R}_{-\theta(t)} x).
\end{eqnarray*}
We have used the symbol $\langle, \rangle$ to denote the usual
scalar product in the plane and in the second identity the fact that
rotations preserve the scalar product. A simple calculation yields
\begin{eqnarray*}
\partial_t\omega(t,x)&=&-\dot{\theta}(t)\Big\{\big(-x_2\partial_1+x_1\partial_2\big)\omega_0\Big\}(\bold{R}_{-\theta(t)} x) \\
&=&-\dot\theta(t)( x^\perp\cdot
\nabla\omega_0)(\bold{R}_{-\theta(t)} x).
 \end{eqnarray*}
Consequently, the vorticity equation becomes
\begin{equation}\label{vortds}
\Big(v_0(x)-\dot\theta(t) x^\perp\Big)\cdot\nabla\omega_0(x)=0.
\end{equation}
Recall that for a smoothly bounded domain $D$
$$
\nabla{\chi_{D}}=-\vec{n} \,d\sigma,
$$
where $d\sigma$ is the arc-length measure on $\partial D$ and
$\vec{n}$ the exterior unit normal. Then, for an initial vorticity
as in \eqref{multiconss},  we get
$$
-\nabla
\omega_0=\alpha_1\vec{n}\,d\sigma_1+\sum_{j=1}^{n-1}(\alpha_{j+1}-\alpha_j)
\vec{n}\,d\sigma_{j+1},
$$
with $d\sigma_j$ the the arc-length measure on the curve $\Gamma_j.$
Since  by the assumption $\alpha_1$ and $\alpha_{j+1}-\alpha_j$ do
not vanish, equation \eqref{vortds} is equivalent to
\begin{equation}\label{second2}
\big(v_0(x)-\dot\theta(t) x^\perp\big)\cdot \vec{n}(x)=0,\quad  x\in
\Gamma_j, \quad 1 \le j \le n.
\end{equation}
We conclude from \eqref{second2} that the only way in which
$\dot\theta(t)$  may be non-constant is that $x^\perp \cdot
\vec{n}(x)=0$ on the union of the interfaces $\Gamma =
\cup_{j=1}^n\Gamma_{j}.$ If this is the case then the interfaces
must be concentric circles. Denote by $z^\prime$ a tangent vector at
the point $x=z$ of $\Gamma$. Using the identities
$$v_0\cdot\vec{n}=\nabla\Psi\cdot z^\prime= 2\hbox{Re}\big(\partial_z\Psi \,z^\prime\big)\quad \hbox{and}\quad x^\perp \cdot\vec{n}
=\hbox{Re}\big(\overline{z}\, z^\prime\big)
$$
we finally obtain
$$
2\hbox{Re}\big(\partial_z\Psi
\,z^\prime\big)=\dot\theta(t)\hbox{Re}\big(\overline{z}\,
z^\prime\big),\quad  z\in \Gamma,
$$
which is \eqref{eq1} after setting $\dot\theta(t) \equiv \Omega$.

We close this subsection by noticing that a rotating vortex appears
as a stationary solution for the vorticity equation for the Euler
system in the presence of the linear external velocity
$v_{e}=\dot\theta x^\perp$, namely,
$$
\partial_t \omega+(v-\dot\theta(t) x^\perp)\cdot\nabla\omega=0.
$$
This follows easily from \eqref{vortds}.

\subsection{The role of the Cauchy transform}
In this subsection we describe the motion of a rotating vortex patch
of the form \eqref{multicon11} by means of the Cauchy transforms of
the domains $D_j.$ Without loss of generality and in order to
simplify the presentation we shall restrict our attention to the
case of two interfaces. Hence the initial vorticity has the form
$$
\omega_0=\chi_{D_1}+(\alpha-1)\chi_{D_2}.
$$
One usually defines the stream function as the logarithmic potential
of the vorticity at time $t$, that is,
\begin{equation}\label{logpot}
\Psi_t(z)=\frac{1}{2\pi}\int_{\RR^2}\omega(t,\xi)
\log|z-\xi|\,dA(\xi), \quad z \in \mathbb{C},
\end{equation}
where $dA$ is Lebesgue measure on the plane. Differentiating
\eqref{logpot} with respect to the variable $z$ yields
\begin{equation}\label{stream}
\begin{split}
\partial_z\Psi(z)&=\frac{1}{4\pi}\int_{D_1}\frac{1}{z-\xi}dA(\xi)+(\alpha-1)\frac{1}{4\pi}\int_{D_2}\frac{1}{z-\xi}dA(\xi)\\
&= \frac14 \mathcal{C}(\chi_{D_1})(z)+(\alpha-1) \frac14
\mathcal{C}(\chi_{D_2})(z),
\end{split}
\end{equation}
where
\begin{equation}\label{Cauchy}
 \mathcal{C}(\chi_D)(z) = \frac{1}{\pi} \int_{D}\frac{1}{z-\xi}dA(\xi), \quad z \in \mathbb{C},
\end{equation}
 denotes the Cauchy transform of the domain $D$ (actually, of the characteristic function of $D$). It is well-known
and easy to check that the Cauchy transform is continuous on
$\mathbb{C}$, holomorphic off $\overline{D}$ and has zero limit at
infinity. If $D$ is a bounded domain with boundary of class $C^1,$
there is a formula for the Cauchy transform of $D$, which we proceed
to describe below,  involving only integrals over the boundary
$\Gamma=\partial D.$  The Cauchy integral of the function
$\overline{z}$ on $D$ is
$$
 \gamma^+(z)=\fint_{\Gamma}\frac{\overline{\xi}}{\xi-z}\,d\xi, \quad z \in
 D,
$$
 where we have used the notation
$\fint_{\Gamma}=\frac{1}{2\pi i}\int_{\Gamma}$. Similarly the Cauchy
integral of $\overline{z}$ on $\mathbb{C}\setminus \overline{D}$ is
$$
 \gamma^-(z)=\fint_{\Gamma}\frac{\overline{\xi}}{\xi-z}\,d\xi, \quad z
 \in \mathbb{C}\setminus \overline{D}.
$$
It is plain that the previous  functions are holomorphic in their
domains of definition. They can be extended continuously up to the
boundary of $D$. This follows easily from dominated convergence and
the identity
$$
\gamma^{\pm}(z)=\fint_{\Gamma}\frac{\overline{\xi}-\overline{z}}{\xi-z}\,d\xi+\overline{z}\,\chi_{D}(z),
$$
which holds in the domains of definition of $\gamma^{\pm}$.
 For the sake of simple notations the one-sided limit at the boundary will be denoted by  $\gamma^{\pm}(z)$ as well.
 The Plemelj-Sokhotski\`i  formulae (see  \cite[p. 143]{V}) for the function $\overline{z}$ are the identities
\begin{eqnarray*}
\gamma^+(z)&=& \hbox{p.v.} \fint_{\Gamma}\frac{\overline{\xi}}{\xi-z}\,d\xi+\frac{\overline{z}}{2}, \quad z \in \Gamma, \\
\gamma^-(z)&=& \hbox{p.v.}
\fint_{\Gamma}\frac{\overline{\xi}}{\xi-z}\,d\xi-\frac{\overline{z}}{2},
\quad z \in \Gamma,
\end{eqnarray*}
where the boundary integrals are understood in the principal value
sense. Subtracting one gets the jump formula
\begin{equation}\label{plem}
\overline{z} = \gamma^+(z)-\gamma^-(z),\quad  z\in \Gamma.
\end{equation}
The Cauchy transform of $D$ can be reconstructed from the functions
$\gamma^{\pm}$. According to the Cauchy-Pompeiu formula for the function
$\overline{z}$ ( see\eqref{GC1} below) one has
\begin{equation}\label{Cauchydins}
\mathcal{C}(\chi_{D})(z)=\overline{z}-\gamma^+(z),\quad  z\in
\overline{D}
\end{equation}
and
\begin{equation}\label{Cauchyfora}
 \mathcal{C}(\chi_{D})(z)=-\,\gamma^{-}(z),\quad  z
\notin {D}.
\end{equation}
 We emphasize that these formulae hold also on the boundary
$\Gamma.$

We now come back to the formula \eqref{stream} for the stream
function. Denote by $\gamma_j^{\pm}$ the Cauchy integrals of
$\overline{z}$ for the domain $D_j, \; i=1,2.$  The identities
\eqref{Cauchydins} and \eqref{Cauchyfora} combined with
\eqref{stream} yield
\begin{equation}\label{gam1}
4\partial_z\Psi(z)=
\overline{z}-\gamma_1^+(z)+(1-\alpha)\gamma_2^-(z), \quad z \in
\Gamma_1
\end{equation}
and
\begin{equation}\label{gam2}
4\,\partial_z\Psi =
\overline{z}-\gamma_1^+(z)+(1-\alpha)\gamma_2^-(z), \quad z\in
\Gamma_2.
\end{equation}

 Putting together \eqref{gam1}, \eqref{gam2}  and \eqref{eq1} we obtain the nonlinear
 system of two equations
\begin{equation}\label{bdd1}
\hbox{Re}\Big\{\Big(\lambda \overline{z}+(1-\alpha)
{\gamma_2^-}(z)-\gamma_1^+(z)\Big)z^\prime \Big\}=0,\quad  z\in
\Gamma_1\cup \Gamma_2,
\end{equation}
with $\lambda=1-2\Omega.$

We mention for future reference that on $\Gamma_2$ the preceding
equation can also be written in the form
\begin{equation}\label{bd1} \hbox{Re}\Big\{\Big(
(\alpha-2\Omega)\overline{z}+(1-\alpha)\gamma_2^+(z)-\gamma_1^+(z)\Big)z^\prime
\Big\}=0,\quad  z\in \Gamma_2.
\end{equation}

\section{Tools from potential theory}
\subsection{Preliminaries on  complex analysis}

We begin by recalling a classical results about complex functions.
The derivatives of a smooth function $\varphi: \CC \to \CC$ with
respect to $z$ and $\overline{z}$ are defined as
$$
\frac{\partial \varphi}{\partial z}=\frac{1}{2}\Big(\frac{\partial
\varphi}{\partial x} -i\frac{\partial \varphi}{\partial y}\Big)
\quad \quad \quad \mbox{and} \quad \quad \quad  \frac{\partial
\varphi}{\partial \overline{z}} =\frac{1}{2}\Big(\frac{\partial
\varphi}{\partial x}+i\frac{\partial \varphi}{\partial y}\Big).
$$

     Let $D$ be a finitely connected domain bounded by finitely many smooth Jordan curves and let $\Gamma$ be $\partial D$ endowed with
     the positive orientation.  Then the Cauchy-Pompeiu formula
     reads as
     \begin{equation*}\label{GC}
    \varphi(z)\chi_{\overline{D}}(z) = \frac{1}{2 \pi i}\int_{\Gamma}\frac{\varphi(\xi)}
    {\xi -z}d\xi-\frac{1}{\pi}\int_{D}\frac{\partial \varphi} {\partial \overline{\xi}}(\xi) \frac{1}{\xi-z}\, dA(\xi)
     ,\quad \, z\in \CC.
     \end{equation*}
For $z\in \partial D$ the boundary integral has to be understood as
the limit from $D$ of the same integral.  Taking
$\varphi(z)=\overline{z}$ we obtain
    \begin{equation}\label{GC1}
 \overline{z}\chi_{\overline{D}}(z) =   \frac{1}{2\pi i}\int_{\Gamma}\frac{\overline{\xi}}{\xi
 -z}d\xi+\mathcal{C}(\chi_D)(z), \quad z \in \CC.
    \end{equation}

     \subsection{Cauchy transform} We intend to compute the Cauchy transform \eqref{Cauchy} of discs and ellipses.
This can be done rather easily using \eqref{Cauchydins} and
\eqref{Cauchyfora}, namely,
 \begin{equation*}
\mathcal{C}(\chi_D)(z)=
 \left\{
\begin{array}{ll}
\overline{z}-\gamma^+(z), \quad \, z\in D \\
-\gamma^-(z), \quad z\notin D.
\end{array} \right.
     \end{equation*}

 $\bullet$ {\it The discs.} To begin with we consider the unit disc.
 Since $\xi \overline{\xi} =1 $ on $\partial D,$
$$
\gamma^+(z)=\fint_{|\xi|=1}\frac{1}{\xi(\xi-z)}d\xi=0, \quad  \,z\in
D\quad\hbox{and} \quad  \gamma^-(z)=-\frac1z, \quad  z\notin
\overline{D}.
$$
Therefore
\begin{equation*}
\mathcal{C}(\chi_D)(z)=
 \left\{
\begin{array}{ll}
\overline{z}, \quad z\in D \\
\frac1z, \quad z\notin D.
\end{array} \right.
\end{equation*}
 For a disc of center  $z_0$ and radius $r$  translating and
 dilating the previous result gives
     \begin{equation*}
\mathcal{C}(\chi_D)(z)=
 \left\{
\begin{array}{ll}
\overline{z}-\overline{z}_0, \quad  z\in D \\
\frac{r^2}{z-z_0}, \quad z \notin D.
\end{array} \right.
     \end{equation*}

$\bullet$ {\it The  ellipses.}  Let $D$ be the domain enclosed by
the ellipse $\{\frac{x^2}{a^2}+\frac{y^2}{b^2} = 1\}.$  We set
$c^2=a^2-b^2.$ If the major semi-axis is $a$ then the foci of the
ellipse are $\pm \sqrt {a^2 -b^2}.$  Otherwise the foci are $ \pm i
\sqrt {b^2 -a^2}.$  Rewriting the cartesian equation of the ellipse
in terms of the variables $z$ and $\overline{z}$ and solving for
$\overline{z}$ leads to
     $$
\overline{z}=Q z+F(z), \quad z\in \partial D,
     $$
     where
     $$
     Q=\frac{a-b}{a+b}\quad\quad \quad {\hbox{and}\quad \quad \quad F(z)=\frac{2ab}{z\Big(1+\sqrt{1-\frac{c^2}{z^2}}\Big)}}\cdot
     $$
     By Cauchy's Integral Formula
     $$
     \fint_{\partial D}\frac{Q\xi}{\xi-z}d\xi=Qz,\,  z\in D.
     $$
     Since $\xi\mapsto \frac{F(\xi)}{\xi-z}$ is holomorphic off $\overline{D}$ and has a double zero at
     infinity,
     $$
      \fint_{\partial D}\frac{F(\xi)}{\xi-z}d\xi=0,\,  z\in D.
     $$
   Hence
 \begin{equation*}\label{gammaplus}
     \gamma^+(z)=Qz,\quad  z\in D.
 \end{equation*}
    To compute the function $\gamma^-$ we use Cauchy's Integral Formula in $\CC \setminus
    \overline{D}$ to get
    \begin{eqnarray*}
    \gamma^-(z)&=&Q\fint_{\partial D}\frac{\xi}{\xi-z}d\xi+\fint_{\partial D}\frac{F(\xi)}{\xi-z}d\xi\\
    &=& \fint_{\partial D}\frac{F(\xi)}{\xi-z}d\xi = -F(z).
    \end{eqnarray*}

Therefore
     \begin{equation}\label{Eqq4}
     \gamma^+(z)=Qz,\quad z\in D, \quad \quad \quad \gamma^-(z)= -F(z),\quad z\notin
     {D},
     \end{equation}
and
\begin{equation}\label{cauc12}
\mathcal{C}(\chi_D)(z)=
 \left\{
\begin{array}{ll}
\overline{z}-Q z,  \quad z\in D \\
\frac{2ab}{z\Big(1+\sqrt{1-\frac{c^2}{z^2}}\Big)},   \quad z \notin
D.
\end{array} \right.
\end{equation}

     Remark that ${\gamma^-}$ satisfies the equation
     \begin{equation}\label{eq4}
     c^2 \{\gamma^-(z)\}^2+4ab\, z\,{\gamma^-}(z)+4 a^2b^2=0.
     \end{equation}
    For the general case where the ellipse is centered   at $z_0$ and its major axis makes an angle $\theta$ with the horizontal
    axis one has
     \begin{equation*}\label{expr12}
     \gamma^+(z)=e^{-2i\theta}Q (z-z_0)+\overline{z_0}, \quad z\in {D}.
     \end{equation*}

\subsection{Inverse problems}\label{inver}
We shall see along this paper that the equations governing the
interfaces of the rotating patch can be solved in some cases and
allow to get explicitly the Cauchy transforms of the involved
domains. It is a general fact that the knowledge of the Cauchy
transform  outside a domain $D$ is equivalent to the  knowledge of
the geometrical moments $(m_n)_{n\in\N}$ defined by
$m_n=\frac{1}{\pi}\int_{D}z^n dA(z)$ since it is a generating
function of these moments. Now the problem  is to see whether  the
shape of the domain is encoded by its Cauchy transform. This is an
inverse problem  of  potential theory which appears in several
contexts like celestial mechanics or geophysics: earth's shape,
gravitational lensing \cite{Fas}, Hele-Shaw flows \cite{Var},... The
inverse  problem is not uniquely solvable in general as some
counter-examples show (see for instance \cite{Var}). However
uniqueness can be established for example under the assumption that
the domains are starlike with respect to a common point
(\cite{Nov}).


In our context, the Cauchy transform has a special algebraic form
and, as we shall see, this determines uniquely the shape giving rise
to the Cauchy transform at hand. There are two kinds of problems
that we are led to deal with. In the first one we discuss the case
where the Cauchy transform is known inside the domain and given by a
first order polynomial function. In the second one the Cauchy
transform is known outside the domain and this case seems to be
trickier. For the first case we prove the following result.

\begin{prop}\label{propinv}
Let $\Gamma$ be a Jordan curve of class $C^1$ enclosing a bounded
domain $D$. Assume that there exist $Q\in \RR$ and $z_0\in \CC$ such
that
$$
\gamma^+(z)\triangleq \frac{1}{2\pi
i}\int_{\Gamma}\frac{\overline{\xi}}{\xi-z}d\xi=Q(z-z_0)+\overline{z_0},\quad
 z\in D.
$$
Then the curve $\Gamma$ is an ellipse of center $z_0$ with semi-axes
$a$ and $b$ satisfying
 \begin{equation*}
Q= \frac{a-b}{a+b}\cdot
     \end{equation*}
\end{prop}

\begin{rema}
It is a surprising consequence of the Proposition that $|Q|$ must be
strictly less than one. In other words, it is not possible to find a
Jordan curve satisfying the hypotheses of Proposition \ref{propinv}
with $|Q|\geq 1.$
\end{rema}
\begin{proof}
According to the jump formula \eqref{plem} we have on  $\Gamma$ the
decomposition
\begin{eqnarray*}
\overline{z}&=& \gamma^+(z)- \gamma^-(z)\\
&=& Q(z-z_0)+\overline{z_0}- \gamma^-(z),
\end{eqnarray*}
with $\gamma^-$ holomorphic on $\CC_\infty\backslash \overline{D}$
and decaying at infinity like $\frac{1}{z}.$ It follows that for any
$z\in \Gamma$
$$
(\overline{z}-\overline{z_0})^2+(z-{z_0})^2=(1+Q^2)(z-z_0)^2-2Q
(z-z_0)\gamma^-(z)+\{\gamma^-(z)\}^2
$$
and
$$
|z-z_0|^2=Q(z-z_0)^2-(z-z_0)\gamma^-(z).
$$
Let $A$ and $B$ be  two real numbers that will be chosen later on.
For  $z\in \Gamma$ one has
$$
-A\Big((\overline{z}-\overline{z_0})^2+(z-{z_0})^2\Big)+B|z-z_0|^2=\big(
B Q-A(1+Q^2)\big)(z-z_0)^2+ g(z),
$$
with
$$
g(z)\triangleq \big(2AQ-B\big)
(z-z_0)\gamma^-(z)-A\{\gamma^-(z)\}^2.
$$
Now choose $A$ and $B$ such that $ BQ-A(1+Q^2)=0$ in order to kill
the quadratic term. For example we can take
$$
A=Q\quad\hbox{and}\quad B=1+Q^2.
$$
Hence
$$
-Q\Big((\overline{z}-\overline{z_0})^2+(z-{z_0})^2\Big)+(1+Q^2)|z-z_0|^2=g(z),\quad
 \, z\in \Gamma.
$$
The function $g$ is clearly holomorphic on $\CC\backslash
\overline{D}$ and has a limit at infinity given by
\begin{eqnarray*}
\lim_{z\to\infty}g(z)&=&(2AQ-B)\lim_{z\to \infty}\frac{1}{2 \pi i}\int_{\Gamma}\frac{\overline{\xi} z}{\xi-z}d\xi\\
&=& (1-Q^2) \frac{1}{2 \pi i}\int_{\Gamma}{\overline{\xi} }\,d\xi \\
&=& \frac{1-Q^2}{\pi} |D|,
\end{eqnarray*}
where we applied Green-Stokes in the last identity. Notice that $g$
has a continuous extension up to the boundary $\Gamma$ and takes
real values on this set. Then the imaginary part   of $g$ is a
harmonic function on the exterior domain $\CC\backslash
\overline{D}$, continuous up to the boundary and satisfying
$$
 \hbox{Im } g(z)=0, \quad z\in \Gamma \quad\hbox{and}\quad \lim_{z\to
\infty}\hbox{Im } g(z)=0 .
$$
By the maximum principle we conclude that $\hbox{Im } g$ is
identically zero on $\mathbb{C} \setminus \overline{D}$. Thus the
holomorphic function $g$ is real on $ \mathbb{C} \setminus
\overline{D}$ and consequently must be constant. This means that
$$
-Q\Big((\overline{z}-\overline{z_0})^2+(z-{z_0})^2\Big)+(1+Q^2)|z-z_0|^2=C,\quad
z\in \Gamma
$$
with $C$ a constant.  Set $X=\hbox{Re}(z-z_0)$ and $ Y=\hbox{Im
}(z-z_0)$ then
$$
(1-Q)^2 X^2+(1+Q)^2 Y^2=C, \quad \hbox{on}\quad \Gamma.
$$
This is an equation  for the curve $\Gamma$ in the cartesian
coordinates $X$ and $Y$. For $Q\not\in\{-1,1\}$ the curve $\Gamma$
is an ellipse. For $Q \in \{-1,1\}$  the curve reduces to a segment,
which is not possible by the assumptions. The proof of the desired
result is complete.
\end{proof}
Next we shall consider the case where the Cauchy transform is
prescribed outside the domain. We will   prove the following result.
\begin{prop}\label{invext}
Let $\Gamma$  be a Jordan curve of class $C^1$ enclosing a domain
$D$ and let $z_1$ be a point in $D$ such that
$$
\frac{1}{2 \pi
i}\int_{\Gamma}\frac{\overline{\xi}}{\xi-z}d\xi=\frac{a}{z-z_1}+\frac{b}{(z-z_1)^2},\quad
z\notin \overline{D},
$$
with $a$ and $b$ real constants. Then there exists a constant $c$
such that curve $\Gamma$ is contained in the set
$$
|z-z_1|^4+a|z-z_1|^2+2b\,\hbox{Re} \,z=c.
$$
\end{prop}
\begin{proof} From the jump formula \eqref{plem} we have
$$
\overline{z}=\gamma^+(z)-\gamma^-(z), \quad z \in \Gamma,
$$
with
$$
 \gamma^+(z)= \frac{1}{2 \pi i}
\int_{\Gamma}\frac{\overline{\xi}}{\xi-z}\,d\xi, \quad z\in D, \quad
\hbox{and} \quad  \gamma^-(z)=\frac{1}{2 \pi i}
\int_{\Gamma}\frac{\overline{\xi}}{\xi-z}\,d\xi, \quad z\notin
\overline{D}.
$$
By assumption
\begin{equation*}\label{gamma12}
\overline{z}=\gamma^+(z)-\frac{a}{z-z_1}-\frac{b}{(z-z_1)^2}, \quad
z\in \Gamma.
\end{equation*}
Set $w=z-z_1,\,\,\widetilde{\Gamma}\triangleq\Gamma -z_1 \,$ and \,
$\widetilde{D}=D-z_1$.  Then the preceding identity can be written
as
\begin{equation}\label{sifr}
\overline{w}=\phi(w)-\frac{a}{w}-\frac{b}{w^2};\quad  w\in
\widetilde{\Gamma},
\end{equation}
with $\phi(w)\triangleq \gamma^+(z_1+w)- z_1 $, which is holomorphic
in $\widetilde{D}$. By \eqref{sifr}
\begin{equation*}\label{wahed}
w+\overline{w}=w+\phi(w)-\frac{a}{w}-\frac{b}{w^2}, \quad w \in
\widetilde{\Gamma},
\end{equation*}
\begin{equation*}\label{ethnan}
w\overline{w}=w\phi(w)-a-\frac{b}{w}, \quad w \in
\widetilde{\Gamma},
\end{equation*}
and
\begin{equation*}\label{ethnan2}
(w\overline{w})^2=\frac{b^2}{w^2}+\frac{2ab}{w}+\phi_1(w), \quad w
\in \widetilde{\Gamma},
\end{equation*}
where $\phi_1$ is holomorphic in $\widetilde{D}$. Taking the
appropriate linear combination of the previous three identities we
kill the singularity at the origin, that is,
\begin{equation*}\label{ethnan3}
(w\overline{w})^2+aw\overline{w}+b(w+\overline{w})=b(w+\phi(w)+aw\phi(w)-a^2+\phi_1(w)\triangleq
\phi_2(w), \quad  w\in \widetilde{\Gamma}.
\end{equation*}
It is plain that $\phi_2$ is holomorphic in $\widetilde{D}$ and the
function in the left-hand side is real-valued. Thus $\phi_2$ is
constant and so
\begin{equation}\label{cart}
|w|^4+a|w|^2+b(w+\overline{w})= c, \quad  w\in\,\widetilde{\Gamma}.
\end{equation}
This completes the proof.
\end{proof}

\section{Proofs of the main results}
In this section we  prove Theorem \ref{impresuld} and Theorem
\ref{impresul}. Recall that we are dealing with  a rotating
vorticity of the form \eqref{multiconss} with only two interfaces,
that is,
\begin{equation*}\label{formza}
\omega_0=\chi_{D_1}+(\alpha-1)\chi_{D_2}, \quad \alpha\in \R,
\end{equation*}
where $D_1$ and $D_2$ are simply connected domains satisfying
$\overline{D_2}\subset D_1.$  As we have already seen the
description  of the rotating vorticity in this special case is
governed by  the \mbox{equations  \eqref{bdd1}.} Owing to the
complicated structure of this  system, which  is strongly nonlinear
and nonlocal,  a  description of the full set of solutions seems to
be out of reach. However, as we stated in Theorem 1 we can show that
if one of the interfaces is a circle then the patch is necessarily
trivial, in the sense that it is an annulus.  In Theorem 2 we
completely solve the system assuming that the inner interface is an
ellipse. Then the exterior interface is a confocal ellipse and
certain relations (introduced in \cite{Flierl}) between the angular
velocity of rotation, the inner vorticity $\alpha$ and the
parameters of the ellipses must be satisfied. Likewise the result
should also hold under the assumption that the exterior interface is
an ellipse, but we have not been able to solve the corresponding
inverse problem.

\subsection{Circular interfaces : the proof of Theorem 1}


\begin{proof}[Proof of Theorem \ref{impresuld}]
The proof relies on equation \eqref{bd1} combined with the inverse
problem results established in the previous section. We first study
the case in which the inner interface $\Gamma_2$ is a circle, which
is easier than the other one.
 \vspace{0,2cm}

 {\bf Case 1 :
$\Gamma_2$ a circle.} \vspace{0,2cm}

Let $\Gamma_2$ be a circle centered at $z_2.$  Assume,  without loss
of generality, that $z_2$ lies in the real axis. Then
$\gamma_2^+(z)={z}_2$ and thus equation \eqref{bd1} reduces to
$$
\hbox{Re}\Big\{\Big((\alpha-2\Omega)\overline{z}+(1-\alpha){z}_2-\gamma_1^+(z)\Big)
z^\prime\Big\}=0, \quad z \in \Gamma_2.
$$
Since $z^\prime=i(z-z_2)$ is a tangent vector at the point $z\in
\Gamma_2$,
$$
\hbox{Re}\Big\{\Big((\alpha-2\Omega)\overline{z}+(1-\alpha){z}_2-\gamma_1^+(z)\Big)i(z-z_2)\Big\}=0,
\quad z \in \Gamma_2.
$$
Observe that $\gamma_1^+(z)=(1-2\Omega) {z}_2$ is a solution of the
above equation. We will show that this is the only solution. Set
$\varphi(z)=\gamma_1^+(z)-(1-2\Omega) z_2,$ so that
$$
\hbox{Im}\big\{\varphi(z)(z-z_2)\big\}=0,\quad  z\in \Gamma_2.
$$
It is plain that $z\mapsto \varphi(z)(z-z_2)$ is holomorphic in
$D_1$, which contains $\overline{D}_2$, and its imaginary part is a
harmonic function vanishing on the boundary $\Gamma_2$. By the
maximum principle  $\hbox{Im}\big\{\varphi(z)(z-z_2)\big\}=0$ in
$D_2$ and so  $\varphi(z)(z-z_2)$ is constant in $D_2$. Evaluating
at $z_2$ we see that this constant must be zero and then that
$\varphi$ vanishes identically on $D_2$ and hence on $D_1$ by
holomorphic continuation. Therefore
\begin{equation}\label{gamma2}
\gamma_1^+(z)=(1-2\Omega) {z}_2,\quad  z\in D_1.
\end{equation}
Now in view of Proposition \ref{propinv} the function $\gamma_1^+$
determines the shape of the boundary $\Gamma_1,$  which turns out to
be a circle centered at the point $z_1\triangleq(1-2\Omega){z}_2.$
The next step is to show that the two circles have the same center.
With this in mind we substitute in equation \eqref{bdd1} the
expression \eqref{gamma2}  for $\gamma_1^+$
 and the identity
$$
\gamma_2^-(z)=\frac{r^2}{z_2-z},  \quad z \notin D_2,
$$
$r$ being the radius of the circle $\Gamma_2.$  We conclude that
$$
\hbox{Im}\Big\{\Big(\lambda(\overline{z}-{z}_2)+(1-\alpha)\frac{r^2}{z_2-z}\Big)(z-z_1)\Big\}=0,\quad
z\in \Gamma_1.
$$
Since $z_1$ and $z_2$ are real,
$$
\lambda({z}_1-{z}_2)\,\hbox{Im}(z-z_1)=(1-\alpha)r^2\,\hbox{Im}\Big\{\frac{z-z_1}{z-z_2}\Big\},\quad
z\in \Gamma_1.
$$
Set $w=z-z_1$ and  $z_0=z_1-z_2.$  Write the preceding equation in
terms of $w$ and $z_0,$  replace $w$ by $-w$ and add the two
equations. We obtain
$$
0=(1-\alpha) r^2\hbox{Im}\big\{\frac{w^2}{w^2-z_0^2}\big\},\quad
  w \in -z_1+\Gamma_1.
$$
As $\alpha\neq 1$ we obtain  $z_0=0$, namely $z_1=z_2.$  Then the
interfaces are concentric circles and there is no restriction on the
parameters $\alpha$ and $\Omega$. This is coherent with the fact
that in this case the vorticity is radial and therefore the flow is
stationary.

\vspace{0,2cm}

{\bf{ Case 2 : $\Gamma_1$  a circle.}} \vspace{0,2cm}

Let $z_1$ be the center of $\Gamma_1.$ Without loss of generality we
may assume that $z_1\in\RR$ and that the center  of mass is the
origin. This implies that the center of mass $z_2$ of $D_2$ is real.
Our goal is to prove  that $\Gamma_2$ must be a circle centered at
$z_1.$  Equation \eqref{bdd1} on $\Gamma_1$ takes the form
\begin{equation*}
\hbox{Im}\Big\{\Big(\lambda \overline{z}-{z}_1+(1-\alpha)
{\gamma_2^-}(z)\Big)(z-z_1) \Big\}=0,\quad  z\in \Gamma_1,
\end{equation*}
which is clearly equivalent to
\begin{equation*}
\hbox{Im}\Big\{\Big((\lambda-1){z}_1+(1-\alpha)
{\gamma_2^-}(z)\Big)(z-z_1) \Big\}=0,\quad  z\in \Gamma_1.
\end{equation*}
Set $w=z-z_1$, $\tilde{\Gamma}_j = -z_1+\Gamma_j$ and $\tilde{D_j}=
-z_1 + D_j, i=1,2.$  Thus the preceding equation becomes
$$
\hbox{Im}\Big\{\Big((\lambda-1){z}_1+(1-\alpha)
{\gamma_2^-}(z_1+w)\Big) w \Big\}=0,\quad  w\in \tilde{\Gamma}_1,
$$
Since  $|\frac{\xi}{w}|<1$ for $\xi\in \tilde{\Gamma}_2$ and $w\in
\tilde{\Gamma}_1,$  one has
\begin{eqnarray*}
 {\gamma_2^-}(z_1+w)&=&\frac{1}{2\pi i}\int_{\tilde{\Gamma}_2}\frac{\overline{\xi}}{\xi-w}d\xi\\
&=& \sum_{n \ge 0} \frac{a_n}{w^{n+1}},
\end{eqnarray*}
where
$$
a_n = - \frac{1}{2\pi i}\int_{\tilde{\Gamma}_2} \xi^n
\overline{\xi}\,d\xi.
$$
 Therefore
$$
\Big((\lambda-1){z}_1+(1-\alpha) {\gamma_2^-}(z_1+w)\Big)
w=(\lambda-1){z}_1 w+(1-\alpha) \sum_{n\ge 0} \frac{a_n}{w^{n}}
$$
and
\begin{equation}\label{four1}
\hbox{Im}\Big\{(\lambda-1){z}_1 w+(1-\alpha) \sum_{n \ge 0}
\frac{a_n}{w^{n}}\Big\}=0,\quad  w\in \tilde{\Gamma}_1.
\end{equation}
By Green-Stokes
$$
a_0=\frac{1}{2\pi
i}\int_{\widetilde{\Gamma}_2}\overline{\xi}d\xi=\frac{1}{\pi}|D_2|.
$$
Now since $\alpha\neq 1$ and $a_0\in \RR$,  equation \eqref{four1}
holds true if and only if
$$
(\lambda-1) z_1 r_1^2=(1-\alpha) a_1 \quad\hbox{and}\quad  a_n =0,
\quad n \ge 2,
$$
$r_1$ being the radius of $\Gamma_1.$ Thus we obtain the following
expression for $\gamma_2$
\begin{equation}\label{gamma102}
\gamma_2^-(z)=\frac{a_0}{z-z_1}+\frac{a_1}{(z-z_1)^2}, \quad  z\in
\Gamma_1.
\end{equation}
 Since $\gamma_2^-$ is continuous in $\CC\backslash D_2,$  the
pole $z_1$ of  $\gamma_2^-$ must be in $D_2.$ Now we will evaluate
the coefficient $a_1$. Using Green-Stokes
\begin{eqnarray*}
a_1&=&\frac{1}{2\pi i}\int_{-z_1+\Gamma_2}|\xi|^2d\,\xi\\
&=&\frac{1}{\pi}\int_{-z_1+D_2}\xi\,dA(\xi)\\
&=&\frac{1}{\pi}\int_{D_2}{(\xi-z_1)}\,dA(\xi).
\end{eqnarray*}
Let $z_2$ be the center of mass of $D_2.$ Thus
 $$
 \int_{D_2}{(\xi-z_2)}dA(\xi)=0,
 $$
 and so
 $$
 a_1=\frac1\pi(z_2-z_1)|D_2|.
 $$
 If we knew that $z_1=z_2$, then
 $a_1=0$ and therefore $\Gamma_2$ would be a circle of center $z_1$ by Proposition \ref{invext}.

  It remains to show that  $a_1$ vanishes.
Combining equations \eqref{bdd1}  and  \eqref{gamma102} we get
\begin{equation}\label{bdr1}
\hbox{Re}\Big\{\Big((\lambda-1){z_1}+
\lambda\overline{w}+\frac{a_0(1-\alpha)}{w}+\frac{a_1(1-\alpha)}{w^2}\Big)w^\prime
\Big\}=0,\quad w\in \tilde{\Gamma}_2,
\end{equation}
Our next task is to find a useful expression for a tangent vector
$w^\prime$ to  $\tilde{\Gamma}_2$ at the point $w$. Recall that by
Proposition \ref{invext} the  curve $\tilde{\Gamma}_2$ is defined in
Cartesian coordinates by
$$
P(x,y)\triangleq (x^2+y^2)^2+a_0(x^2+y^2)+2a_1 x=c.
$$
A tangent vector is then given by
\begin{eqnarray*}
w^\prime&=&-\partial_yP+i\,\partial_x P\\
&=&4|w|^2 iw+2ia_0w+2ia_1.
\end{eqnarray*}
Substituting this expression for $w'$ in equation \eqref{bdr1} one
gets
$$
\hbox{Im }\Big\{(\lambda-1) z_1 w(2|w|^2+a_0)+\lambda
a_1\overline{w}+\frac{2A}{w}(|w|^2+a_0)+\frac{a_1\,A}{w^2}\Big\}=0,
\quad w\in \tilde{\Gamma}_2,
$$
where we have set $A= a_1(1-\alpha)$. This gives
$$
(1-\lambda) z_1 (2|w|^2+a_0)+\lambda
a_1+\frac{2A}{|w|^2}(|w|^2+a_0)+a_1\,A\frac{w+\overline{w}}{|w|^4}=0,
\quad w\in \tilde{\Gamma}_2,
$$
which is equivalent to
$$
2(1-\lambda) z_1|w|^6+\big(2A+\lambda a_1+(1-\lambda) a_0
z_1\big)|w|^4+2a_0 A|w|^2+a_1\,A(w+\overline{w})=0,\quad  w\in
\widetilde{\Gamma}_2.
$$
Using \eqref{cart} with $a$ replaced by $a_0$ and $b$ by $a_1$ we
find
\begin{equation}\label{equa}
2(1-\lambda) z_1|w|^6+\big(A+\lambda a_1+(1-\lambda) a_0
z_1\big)|w|^4+a_0 A|w|^2+A c=0,\quad  w\in \widetilde{\Gamma}_2.
\end{equation}
Since $\widetilde{\Gamma}_2$ is connected, either $|w|$ is constant
on $\widetilde{\Gamma}_2$  and we are done, or   $|w|$ takes a
continuum of values on $\widetilde{\Gamma}_2.$ In the second case
the polynomial  obtained by replacing in the left-hand side of
\eqref{equa} $|w|$  by the real variable $t$ has infinitely many
zeroes and hence is the zero polynomial. Thus the coefficient $a_0 A
$ must be zero. Since $\pi a_0 = |D_2|\neq 0$ and $\alpha \neq 1$,
$a_1 =0$ and the proof is complete.
\end{proof}

\subsection{Elliptical interfaces : the proof of Theorem 2}
We turn now to the proof  of Theorem \ref{impresul}.  First we will
prove that if the interior curve is an ellipse then the rigid motion
of the interfaces will force the domains to have the same center of
mass. The equations \eqref{bdd1} and the explicit form of the
function $\gamma_2^+$ will lead to the identification of
$\gamma_1^+,$  via the maximum principle. At this stage we are led
to understand the link between the geometry of the domain and its
inside Cauchy transform $\gamma^+$. This is a kind of inverse
problem of two-dimensional potential theory  that we have already
discussed in the previous section in the context of circles. In the
case at hand we show that the exterior curve $\Gamma_1$ is an
ellipse and we find some information on its shape. Armed with this
precious information we solve explicitly the \mbox{equations
\eqref{bdd1}. }

\subsubsection{{\bf First reduction}}

\begin{Lemm}\label{gamma1ellipse}
Assume that $\omega_0=\chi_{D_1}+(\alpha-1)\chi_{D_2}, \;\alpha \in
\mathbb{R},$ is a rotating vorticity around the origin and that
$\Gamma_2$ is an ellipse. Let $z_2$ be the center of $D_2$. If
$z_2\neq 0$ then the line through the origin and $z_2$ is an axis of
the \mbox{ellipse $\Gamma_2.$} Moreover  $\Gamma_1$  is an ellipse.
\end{Lemm}

\begin{proof}
First of all we can assume without loss of generality that $z_2$ is
a positive real number. Let $\theta$ denote the angle between the
major axis of the ellipse $\Gamma_2$ and the real axis. We have to
show that $2\theta\equiv0 \,[\pi]$.

Recall that by \mbox{\eqref{bd1}} the equation that describes
rotation with angular velocity $\Omega$ on $\Gamma_2$ is
\begin{equation*}
\hbox{Re}\Big\{\Big(
(\alpha-2\Omega)\overline{z}+(1-\alpha)\gamma_2^+(z)-\gamma_1^+(z)\Big)z^\prime
\Big\}=0,\quad  z\in \Gamma_2,
\end{equation*}
where  $z^\prime$ denotes a tangent vector to $\Gamma_2$ at the
point $z$.  For $j\in\{1,2\},$ let $\phi_j$ denote a complex
primitive of $\gamma_j^+$ on the domain $D_j$. This primitive is
well-defined since  $\gamma_j^+$ is holomorphic  on the simply
connected domain $D_j$. Consequently equation \eqref{bd1} is
equivalent to
\begin{equation}\label{bd2}
\big(\frac{\alpha}{2}-\Omega\big)|z|^2+\hbox{Re}\Big((1-\alpha)\phi_2(z)-\phi_1(z)\Big)=C,
\quad  z\in \Gamma_2,
\end{equation}
for some constant $C$.

Let $\psi$ be the solution of the Dirichlet problem on $D_2$ with
boundary data $|z|^2,\; z\in \Gamma_2.$ Since $\psi$ is harmonic in
$D_2$ and $D_2$ is simply connected, there exists a holomorphic $H$
function on $D_2$ such that $\psi$ is the real part of $H.$ Hence
equation \eqref{bd2} becomes
     $$
     \hbox{Re}\Big(\big(\frac{\alpha}{2}-\Omega\big)H(z)+(1-\alpha)\phi_2(z)-\phi_1(z)-C\Big)=0,\quad  z\in \Gamma_2.
     $$
     The function in the left-hand side of the preceding identity is harmonic in $D_2$ and continuous up to the boundary. By the maximum principle
     this function is identically zero in the domain $D_2.$
     Therefore, since holomorphic functions that take real values on
     a domain are constant,
     $$
     \big(\frac{\alpha}{2}-\Omega\big)H^\prime(z)+(1-\alpha)\phi_2^\prime(z)-\phi_1^\prime(z)=0,\quad  z\in D_2,
     $$
     where prime denotes derivative with respect to $z$.
Hence
\begin{equation}\label{gamma}
     \gamma_1^+(z)=\big({\alpha}-2\Omega\big)\partial_z\psi(z)+(1-\alpha)\gamma_2^+(z),\quad  z\in
     D_2,
     \end{equation}
     because  $2\partial_z \psi (z)=H^\prime(z).$
         This determines completely the function $\gamma_1^+$  in $D_2$ and thus in $D_1,$ by analytic
         continuation. To take full advantage of \eqref{gamma} we need to have explicit expressions for $\gamma_2^+$ and $2\partial_z \psi.$
For $\gamma_2^+$ it is just a matter of applying a translation and a
rotation to \eqref{Eqq4}. We get
\begin{equation}\label{gamma2plus}
\gamma_2^+(z)= Q_2 e^{-2i\theta} (z-z_2)+z_2, \quad z \in D_2.
\end{equation}
For  $\partial_z \psi$ we solve explicitly the Dirichlet problem
defining $\psi$ and then we take a derivative with respect to $z.$
In order  to do so we need to write the equation of the boundary
$\Gamma_2$ in the variables $z$ and $\overline{z}.$
      Consider first  the ellipse $\mathcal{E}= \{(x,y): \frac{x^2}{a^2}+\frac{y^2}{b^2}=1\}.$
      Expressing $x$ and $y$ in terms of $z$ and $\overline{z}$ we
      find for $\mathcal{E}$ the equation
      $$
     A(z^2+\overline{z}^2)+B|z|^2=1,\quad A=\frac14\Big(\frac{1}{a^2}-\frac{1}{b^2}\Big)\quad \hbox{and}\quad B=\frac12\Big(\frac{1}{a^2}+\frac{1}{b^2}\Big).
     $$
Assume now that $a$ and $b$ are the semi-axes of the ellipse
$\Gamma_2.$ Then an equation for $\Gamma_2$ is
     $$
     A\Big(e^{-2i\theta}(z-z_2)^2+e^{2i\theta}(\overline{z}-z_2)^2\Big)+B|z|^2+B {z_2}^2-Bz_2(z+\overline{z})=1.
     $$
     Solving for $|z|^2$ and remarking that the function which gives
     the solution is harmonic in $D_2$ we conclude that
     $$
     \psi(z)=\frac{1}{B}- z_2^2+ z_2(z+\overline{z})- \frac{A}{B}\Big(e^{-2i\theta}(z-z_2)^2+e^{2i\theta}(\overline{z}-z_2)^2\Big).
     $$
and
    \begin{equation}\label{SPD}
    \partial_z\psi(z)=z_2-2\frac{A}{B}e^{-2i\theta}(z-z_2).
    \end{equation}
Inserting \eqref{SPD} and \eqref{gamma2plus} in \eqref{gamma} yields
    \begin{equation}\label{gamma1}
  \gamma_1^+(z)=(1-2\Omega) z_2+ Q_1 e^{-2i\theta} (z-z_2), \quad z \in
     D_1,
    \end{equation}
     where
     $$
     Q_1 = (2\Omega-\alpha) \frac{2A}{B}+(1-\alpha) Q_2.
          $$
We have proved that $\gamma_1^+$ is a first degree polynomial and we
are almost done. First we remark that the assumption that the center
of mass of the initial vorticity is the origin implies that the
centers of mass $z_j$ of $D_j, j=1,2,$ satisfy
\begin{equation}\label{centermass}
z_1 = (1-\alpha) \frac{|D_2|}{|D_1|} z_2.
\end{equation}
In particular $z_1$ is a real number.

It is a general fact that if $D$ is any bounded  domain then
\begin{equation}\label{center}
\int_D \overline{z} \,dA(z) = \int_D \gamma^+(z) \,dA(z).
\end{equation} This follows from \eqref{Cauchydins} and the observation that
$$\int_D \int_D \frac{1}{\zeta-z} \,dA(\zeta)dA(z) =0 $$
because the Cauchy kernel is odd. Taking the mean value of
$\gamma_1^+$ on $D_1$ and using \eqref{gamma1} and \eqref{center}
one obtains
$$
z_1 = (1-2\Omega)z_2 + Q_1 e^{-2\theta i} (z_1-z_2).
$$
Thus $2\theta$  is an integer multiple of $ \pi$ and so the line
through the origin and $z_2$ is an axis of $\Gamma_2.$

We are left with the task of showing that $\Gamma_1$ is an ellipse,
which is easy. Indeed, setting
$$
z_1= \frac{\lambda\pm Q_1}{1\pm Q_1} z_2,
$$
where the minus sign corresponds to $\theta=0$ and the plus sign to
$\theta= \pi/2,$ we obtain, rewriting \eqref{gamma1},
$$
 \gamma_1^+(z)= \pm Q_1 (z-z_1)+ z_1, \quad z \in D_1.
$$
Here the plus sign corresponds to $\theta=0$ and the minus sign to
$\theta= \pi /2.$ Since $z_1$ is real an application of Proposition
\ref{propinv} shows that $\Gamma_1$ is an ellipse, which completes
the proof of \mbox{Lemma \ref{gamma1ellipse}.}
\end{proof}


\subsubsection{\bf{Second reduction}}

We know now that if $\omega_0=\chi_{D_1}+(\alpha-1)\chi_{D_2}$ is a
rotating vorticity and the interior curve is an ellipse, then the
exterior curve is also an ellipse and its axes are parallel to those
of the interior ellipse. However, the only information we have up to
now about the relative position of the center of mass of $\omega_0$
and the centers of the ellipses is that they lie on a straight line.
\begin{Lemm}\label{centersorigin}
Assume that $\omega_0=\chi_{D_1}+(\alpha-1)\chi_{D_2}$ is a rotating
vorticity and that $\Gamma_1$ and $\Gamma_2$ are ellipses. Then the
ellipses are centered at the center of mass of $\omega_0.$
\end{Lemm}
\begin{proof}
Assume, without loss of generality, that the center of mass of
$\omega_0$ is the origin. Let $z_j$ be the center of $\Gamma_j,\,
j=1,2.$  We may also assume that $z_2$ is real. Then $z_1$ is also
real because of \eqref{centermass}. By \eqref{bdd1} the equation
that describes rotation with angular velocity $\Omega$ on $\Gamma_1$
is
$$
\hbox{Re}\Big\{\Big(\lambda \overline{z}+(1-\alpha)
{\gamma_2^-}(z)-\gamma_1^+(z)\Big)z^\prime \Big\}=0,\quad  z\in
\Gamma_1,
$$
with $\lambda=1-2\Omega.$  Since $z_1$ is real, by the first
reduction $ \gamma_1^+(z)= \pm Q_1(z-z_1)+z_1.$ Assume that the sign
in front of $Q_1$ is plus (the same argument will work with the
minus sign). Then
$$
\hbox{Re}\Big\{\Big(\lambda \overline{z}+(1-\alpha)
{\gamma_2^-}(z)-Q_1(z-z_1)-z_1\Big)z^\prime \Big\}=0,\quad  z\in
\Gamma_1.
$$
Setting $w=z-z_1$ the preceding equation becomes
\begin{equation}\label{dsr1}
\hbox{Re}\Big\{\Big(\lambda \overline{w}+(1-\alpha)
{\gamma_2^-}(z_1+w)-Q_1w+(\lambda-1)z_1\Big) i(A_1w+B_1\overline{w})
\Big\}=0, \quad w\in -z_1+\Gamma_1.
\end{equation}
 Here $i(A_1w+B_1\overline{w})$ is the tangent vector to the
ellipse $-z_1+\Gamma_1$ at the point $w$. The expression of $A_1$
and $B_1$ in terms of the length of the semi-axes $a_1$ and $b_1$ of
$\Gamma_1$ can be obtained by using the standard parametrization of
an ellipse. One gets
\begin{equation*}\label{tangentellipse}
A_1 = \frac{a_1^2 +b_1^2}{2 a_1 b_1} \quad\quad \quad
\mbox{and}\quad\quad \quad B_1 = \frac{b_1^2 -a_1^2}{2 a_1 b_1}.
\end{equation*}

The ellipse  $\widetilde{\Gamma}_2=-z_2+\Gamma_2$ is centered at the
origin and its axes lies along the coordinate axes. Set
$$
{h}_2(z)=\frac{1}{2\pi
i}\int_{\widetilde{\Gamma}_2}\frac{\overline{\xi}}{\xi-z}d\xi, \quad
z \in \mathbb{C}\setminus \overline{D_2}.
$$
As we mentioned in \eqref{Eqq4} ${h}_2$ is given by
$$
h_2(z)=-\frac{2a_2 b_2}{z\Big(1+\sqrt{1-\frac{c_2^2}{z^2}}\Big)},
\quad z \in \mathbb{C}\setminus \overline{D_2}.
$$
Translating we see that
$$
{\gamma_2^-}(z)={h}_2(z-z_2), \quad  z \notin \overline{D}_2.
$$
Let  $d = z_1-z_2.$  Then equation \eqref{dsr1} can be rewritten as
$$
\hbox{Re}\Big\{\Big(\lambda \overline{w}+(1-\alpha)
{h}_2(w+d)-Q_1w+(\lambda-1)z_1\Big)i(A_1w+B_1\overline{w}) \Big\}=0,
\quad w\in -z_1+\Gamma_1.
$$
The ellipse $-z_1+\Gamma_1$ is centered at  zero and invariant under
the mapping $w \rightarrow -w$. Therefore writing the above equation
for $-w$ and subtracting both equations yields, for $ w\in
-z_1+\Gamma_1 ,$
\begin{equation}\label{Ez12}
2(\lambda-1)
z_1\hbox{Im}(A_1w+B_1\overline{w})+(1-\alpha)\hbox{Im}\Big\{\Big(
{h}_2(w+d)+ {h}_2(-w+d)\Big)(A_1w+B_1\overline{w})\Big\}=0.
\end{equation}
Denote by  $\mathbb{C}_\infty$ the extended complex plane (or
Riemann sphere). Let $U$ be the domain enclosed by the ellipse $
-z_1+\Gamma_1.$ Our next task is to find a solution to the Dirichlet
problem in the domain $\mathbb{C}_{\infty}\setminus \overline{U}$ in
the Riemann sphere with boundary data
$\operatorname{Im}(A_1w+B_1\overline{w}).$
 By \eqref{cauc12}
 \begin{equation}\label{Ez124}
 \overline{w}=Q_1 w-h_1(w), \quad w\in  -z_1+\Gamma_1,
 \end{equation}
 with
 $$
h_1(w)=-\frac{2a_1 b_1}{w\Big(1+\sqrt{1-\frac{c_1^2}{w^2}}\Big)},
\quad z \in \mathbb{C} \setminus {U}.
$$
Hence
\begin{eqnarray*}
\hbox{Im}({\overline{w}})&=&Q_1\hbox{Im}(w)-\hbox{Im}(h_1(w))\\
&=&-Q_1\hbox{Im}(\overline{w})-\hbox{Im}(h_1(w)), \quad w\in
-z_1+\Gamma_1,
\end{eqnarray*}
and so
$$
\hbox{Im}({\overline{w}})=-\frac{1}{1+Q_1}\hbox{Im}(h_1(w)), \quad
w\in  -z_1+\Gamma_1,
$$
and
\begin{eqnarray*}
\hbox{Im}(A_1w+B_1\overline{w})&=&(B_1-A_1)\hbox{Im}({\overline{w}})\\
&=&\frac{A_1-B_1}{1+Q_1}\hbox{Im}(h_1(w)), \quad w\in -z_1+\Gamma_1.
\end{eqnarray*}
The right-hand side above is harmonic in
$\mathbb{C}_{\infty}\setminus \overline{U}$ and then is the solution
of the Dirichlet problem in $\mathbb{C}_{\infty}\setminus
\overline{U}$ with boundary data the left-hand side. Inserting this
identity into \eqref{Ez12} and using \eqref{Ez124} and the relation
$A_1+B_1 Q_1=1$ one gets
$$
\operatorname{Im}\Big\{\mathcal{A}\,z_1\, h_1(w)+(1-\alpha)\Big(
h_2(w+d)+ h_2(-w+d)\Big)\big(w-B_1 h_1(w)\big)\Big\} =0, \quad w\in
-z_1+\Gamma_1,
$$
where $\mathcal{A}$ stands for $ 2\frac{A_1-B_1}{1+Q_1}(\lambda-1).$
In the left-hand side of the preceding identity one is taking the
imaginary part of a holomorphic function in $\mathbb{C} \setminus
\overline{U}.$  Hence, for some constant $C,$
\begin{equation}\label{hol}
\mathcal{A}\,z_1\,h_1(w)+(1-\alpha)\Big(  {h}_2(w+d)+
{h}_2(-w+d)\Big)\big(w-B_1h_1(w)\big)=C, \quad w \in \mathbb{C}
\setminus \overline{U}.
\end{equation} Observe that
$$
C=(1-\alpha)\lim_{w\to\infty} w\big({h}_2(w+d)+ {h}_2(-w+d)\big)=0.
$$
Computing the coefficient of $\frac{1}{w}$ in the expansion at
$\infty$ of the left-hand side of \eqref{hol} we get the relation
\begin{equation}\label{coe}
\mathcal{A} z_1 a_1 b_1=2(1-\alpha) a_2 b_2 d.
\end{equation}
 Set
 $$
 F_j(w)=\frac{1}{w\Big(1+\sqrt{1-\frac{c_j^2}{w^2}}\Big)},\quad j=1,2, \quad w \in \mathbb{C}
\setminus \overline{U}.
 $$
 An easy argument based on \eqref{hol} and \eqref{coe} gives
\begin{equation}\label{zero}
  2d\, F_1(w)+\Big(F_2(w+d)+F_2(-w+d)\Big)\big(w-c_1^2F_1(w)\big)=0, \quad w \in \mathbb{C} \setminus
\overline{U}.
 \end{equation}
The function in the left-hand side above is odd and holomorphic at
$\infty$ so that in its expansion in powers of $1/w$ the even powers
vanish identically. The coefficient of $1/w$ also vanishes
identically. Instead, the fact that the other odd powers vanish,
because of \eqref{zero}, provides a countable family of equations in
the parameters $d, c_1$ and $c_2.$  Our goal is to show that $d=0$
using the equations corresponding to the coefficients of $1/w^3,
1/w^5$ and $1/w^7$. Recall that then $z_1=z_2=0$ and we are done.

 The  coefficient of $1/w^3$ is
\begin{equation*}\label{c3}
d \, \left(\frac{3}{4}  c_1^2-\frac{3}{4}  c_2^2 -d^2 \right).
\end{equation*}
Hence either $d=0$  or
\begin{equation}\label{di}
d^2 =  \frac{3}{4} \left(  c_1^2- c_2^2\right).
\end{equation}
The coefficient of $1/w^5$ is
\begin{equation}\label{c5}
\frac{1}{8} d \, \left(2c_1^4-20c_2^2d^2-5c_2^4 -8d^4 + 4 c_1^2 d^2
+ 3 c_1^2 c_2^2 \right).
\end{equation}
As before, if  $d=0$ we are done and so we can assume that this is
not the case. If $c_j =0$ for $j=1$ or $j=2,$ then $\Gamma_j$ is a
circle and this case has been dealt with in subsection 5.1.

Dividing in \eqref{c5} by $c_2^4,$ eliminating the even powers of
$d$ by means of \eqref{di} and setting $q=c_1/c_2$ we get $
q^2-12q+11=0,$ which yields $q =1 $ or $q=11.$ If $q=1$ then $d=0$
by \eqref{di}. Let $q=11.$  The coefficient of $1/w^7$ turns out to
be
\begin{equation}\label{c7}
\frac{1}{64} d \, \left(-366 c_2^2 d^4 -280c_2^4d^2-35c_2^6 +9 c_1^6
-64d^6+6 c_1^4 c_2^2 +  80 c_1^2 c_2^2 d^2 + 20 c_1^2 c_2^4 + 8
c_1^2 d^4 \right).
\end{equation}
Eliminating the even powers of $d$ in \eqref{c7} by means of
\eqref{di} and  setting $c_1^2 = 11 c_2^2$ we obtain
\begin{equation*}\label{c72}
\frac{-1450}{64} d \, c_2^6 =0 .
\end{equation*}
Since we are in the case $c_2 \neq 0,$ we conclude that $d=0$, which
completes the proof.
\end{proof}

     \subsubsection{ {\bf Resolution of the boundary equations }}
     Up to now we have shown that if the interior curve is an ellipse then necessarily the exterior curve is an ellipse
      with the same center and parallel axes. Our next target is to give a complete description of the parameters $\lambda,\alpha, Q_1$ and $Q_2$
       in order to get a uniform rotation. This will complete the proof of Theorem
       \ref{impresul}.

       We start by investigating the equation on the interior curve $\Gamma_2$.
\vspace{0,2 cm}

 {\bf Equation on $\Gamma_2.$}
     Recall that the equation \eqref{bd1}
that describes rotation with angular velocity $\Omega$ on $\Gamma_2$
is

           $$
     \hbox{Re}\Big\{\Big( (\alpha-2\Omega)\overline{z}+\big[(1-\alpha) Q_2 -Q_1\big] z\Big)z^\prime  \Big\}=0, \quad z\in \Gamma_2.
     $$
     We have used the fact that $\gamma_j^+(z)= Q_j \,z, \,j = 1, 2.$ As we mentioned before, a straightforward computation
     shows that a tangent vector to the ellipse $\Gamma_j$ at the point $z$ is given by
     $$
     z^\prime=i(A_j z+B_j \overline{z}), \quad A_j=\frac{a_j^2+b_j^2}{2a_j b_j},\quad B_j= \frac{b_j^2-a_j^2}{2a_j b_j}\cdot
     $$
     Recall that $c_j^2 = a_j^2-b_j^2
     $ gives the foci of the ellipse.
     Hence we obtain
     $$
     \hbox{Re}\Big\{ i\Big(\big[(1-\alpha) Q_2 -Q_1\big] A_2 z^2+(\alpha-2\Omega)B_2\overline{z}^2 \Big) \Big\}=0,  \quad z \in \Gamma_2,
     $$
     which is equivalent to
     $$
     \hbox{Re}\Big\{ i\Big(\big[(1-\alpha) Q_2 -Q_1\big] A_2 -(\alpha-2\Omega)B_2 \Big)z^2 \Big\}=0, \quad z\in \Gamma_2,
     $$
     This condition is satisfied only when
     \begin{equation}\label{Eqs2}
    \big[(1-\alpha) Q_2 -Q_1\big] A_2 +(2\Omega-\alpha)B_2=0.
     \end{equation}
     We would to write this equation in terms of $Q_1, Q_2$ and $\Omega$ only. From the elementary  identities
     $$
     A_j+B_j Q_j=1 \quad\hbox{and}\quad A_j^2-B_j^2=1
     $$
     we get
     \begin{equation}\label{Eq8}
     A_j=\frac{1+Q_j^2}{1-Q_j^2}\quad\quad\quad\mbox{and}\quad\quad\quad B_j=\frac{-2Q_j}{1-Q_j^2}.
     \end{equation}
     Thus  \eqref{Eqs2} becomes
     \begin{equation}\label{eq3}
     \big[(\alpha-1) Q_2 +Q_1\big] (1+Q_2^2) =2(\alpha-2\Omega)Q_2.
     \end{equation}

     \vspace{0,2 cm}

     {\bf{Equation on $\Gamma_1.$}}
     Using equation \eqref{bdd1} on $\Gamma_1$, we get
   \begin{eqnarray*}
     \hbox{Re}\big\{i \lambda B_1\overline{z}^2+\Big((1-\alpha){\gamma_2^-}(z)-Q_1 z\Big) i(A_1 z+B_1 \overline{z})
     \big\}=0, \quad z\in \Gamma_1.
     \end{eqnarray*}
     Since $-\hbox{Re}\{i  Q_1 A_1 {z}^2\}= \hbox{Re}\{i Q_1 A_1 \overline{z}^2\},$
    \begin{equation}\label{eqgam1}
 \hbox{Re}\Big\{i (\lambda B_1+Q_1 A_1) \overline{z}^2+(1-\alpha){\gamma_2^-}(z) i(A_1z+B_1 \overline{z})\Big\}=0, \,\quad z\in \Gamma_1.
    \end{equation}
     Let us introduce the function
     $$G(z)= (\lambda B_1+Q_1 A_1) \overline{z}^2+(1-\alpha){\gamma_2^-}(z)(A_1z+B_1 \overline{z}).$$
      Since on  $\Gamma_1$ we have $\overline{z}=Q_1
      z-{\gamma_1^-}(z),$ $G(z)$ can be written as
         $$
     G(z)=(\lambda B_1+A_1 Q_1) \overline{z}^2+(1-\alpha)(A_1+B_1 Q_1) z {\gamma_2^-}(z)-(1-\alpha)B_1 {\gamma_1^-}(z) {\gamma_2^-}(z).
     $$
     Setting $M=\lambda B_1+Q_1 A_1$ and using the identity  $A_1+B_1 Q_1=1$ we find
     $$
     G(z)=M\overline{z}^2+(1-\alpha)z{\gamma_2^-}(z)-(1-\alpha)B_1 {\gamma_1^-}(z) {\gamma_2^-}(z).
     $$
     Thus equation \eqref{eqgam1} on $\Gamma_1$ becomes
     \begin{equation}\label{Eq7}
     \operatorname{Im}\big\{M\overline{z}^2+(1-\alpha)z{\gamma_2^-}(z)-(1-\alpha)B_1 {\gamma_1^-}(z) {\gamma_2^-}(z)\big\}=0,\, \quad z\in \Gamma_1.
     \end{equation}
    The next step is to solve the Dirichlet problem on the domain $\mathbb{C_\infty}\setminus
    \overline{D_1}$ of the Riemann sphere $\mathbb{C_\infty}$ with boundary data $\operatorname{Im} ( \overline{z}^2).$
With this goal in mind recall the identity
   \begin{equation}\label{schwarz}
    \overline{z}=Q_1 z-{\gamma_1^-}(z), \quad z \in \Gamma_1,
   \end{equation}
     where
     $$
{\gamma_1^-}(z)=\frac{-2a_1b_1}{z\Big(1+\sqrt{1-\frac{c_1^2}{z^2}}\Big)},
\quad z \in \mathbb{C} \setminus \overline{D_1}.
     $$
    Squaring \eqref{schwarz} we obtain
     $$
     \overline{z}^2=Q_1^2 z^2-2Q_1z
     {\gamma_1^-}(z)+\{\gamma_1^-(z)\}^2,  \quad z \in \Gamma_1.
     $$
   By   \eqref{eq4}
  \begin{eqnarray*}
   \overline{z}^2&=&Q_1^2 z^2+\frac{Q_1}{2a_1 b_1}\big(c_1^2 \{\gamma_1^-(z)\}^2+4 a_1^2 b_1^2  \big)+\{\gamma_1^-(z)\}^2\\
   &=&Q_1^2 z^2+2Q_1 a_1 b_1+\big(1+\frac{Q_1 c_1^2}{2a_1 b_1}\big)
   \{\gamma_1^-(z)\}^2, \quad z \in \Gamma_1.
  \end{eqnarray*}
   It is easy to check that
   $$
   1+\frac{Q_1 c_1^2}{2a_1 b_1}=\frac{a_1^2+b_1^2}{2a_1 b_1}=A_1.
   $$
   Consequently
   \begin{eqnarray*}
  \operatorname{Im} \,\overline{z}^2&= &Q_1^2\operatorname{Im}\, z^2+A_1\operatorname{Im} \,\{\gamma_1^-(z)\}^2\\
  &=&-Q_1^2\operatorname{Im}\,\overline{z}^2+A_1\operatorname{Im}\,
  \{\gamma_1^-(z)\}^2, \quad z \in \Gamma_1.
   \end{eqnarray*}
   This gives
   $$
    \operatorname{Im}\, \overline{z}^2=\frac{A_1}{1+Q_1^2} \operatorname{Im} \,\{\gamma_1^-(z)\}^2,\quad \quad z\in
    \Gamma_1,
   $$
   which tells us that the function on the right-hand side is the
   solution of the Dirichlet problem in $\mathbb{C_\infty}\setminus
    \overline{D_1}$ with boundary data given by the left-hand side.
   Inserting this into \eqref{Eq7}  yields
   $$
  \operatorname{Im}\,\Big( \frac{M A_1}{1+Q_1^2}  \{\gamma_1^-(z)\}^2+(1-\alpha)z{\gamma_2^-}(z)-(1-\alpha)B_1 {\gamma_1^-}(z)
   {\gamma_2^-}(z) \Big)=0, \quad z\in \Gamma_1.
   $$
Since the function inside the imaginary part in the preceding
identity is holomorphic on $\mathbb{C_\infty} \setminus
\overline{D_1},$ it is constant. In other words, for some constant
$C,$
   $$
   \frac{MA_1}{1+Q_1^2}  \{\gamma_1^-(z)\}^2+(1-\alpha)z{\gamma_2^-}(z)-(1-\alpha)B_1 {\gamma_1^-}(z) {\gamma_2^-}(z)=C,\quad \quad z\in \CC\backslash D_1.
   $$
   In view of \eqref{eq4} we obtain
   $$
    -\frac{MA_1}{1+Q_1^2}  \{\gamma_1^-(z)\}^2+(1-\alpha)B_1 {\gamma_1^-}(z) {\gamma_2^-}(z)+(1-\alpha)\frac{c_2^2}{4a_2 b_2}\{\gamma_2^-(z)\}^2=0,\quad  \quad z\in  \CC\backslash D_1.
   $$
   and recalling that  $B_2 = -c_2^2/2 a_2 b_2 $
   $$
    2\frac{MA_1}{1+Q_1^2}  \{\gamma_1^-(z)\}^2-2(1-\alpha)B_1 {\gamma_1^-}(z) {\gamma_2^-}(z)+(1-\alpha)B_2\{\gamma_2^-(z)\}^2=0,\quad  \quad z\in \CC\backslash D_1.
   $$
Dividing this equation by $\{\gamma_2^-(z)\}^2$  we get a second
degree polynomial equation in the unknown $\gamma_1^- /\gamma_2^-,$
namely,
   $$
   2\frac{MA_1}{1+Q_1^2}  \Big(\frac{{\gamma_1^-}(z)}{{\gamma_2^-}(z)}\Big)^2-2(1-\alpha)B_1 \frac{{\gamma_1^-}(z)}{ {\gamma_2^-}(z)}+(1-\alpha)B_2=0, \quad \quad z\in \CC\backslash D_1.
   $$
   This implies that $\frac{{\gamma_1^-}(z)}{{\gamma_2^-}(z)}=\mu$ with $\mu$ a constant. Consequently,
   $$
   c_1=c_2\quad\hbox{and}\quad \mu=\frac{a_1b_1}{a_2 b_2}.
   $$
   In particular the ellipses $\Gamma_1$ and $\Gamma_2$ are
   confocal. Moreover
   $$
    2\frac{MA_1}{1+Q_1^2}  \big(\frac{a_1b_1}{a_2 b_2}\big)^2-2(1-\alpha)B_1 \frac{a_1b_1}{ a_2b_2}+(1-\alpha)B_2=0.
   $$
   One can easily check that
   $$
    \frac{a_1b_1}{ a_2b_2}=\frac{B_2}{B_1} \quad \quad \mbox{and} \quad\quad -2B_1 \frac{a_1b_1}{ a_2b_2}+B_2=-B_2.
   $$
   This yields
   \begin{equation}\label{Eq10}
   2\frac{MA_1}{1+Q_1^2}B_2=(1-\alpha)B_1^2.
   \end{equation}
   We will rewrite this equation in terms of  $Q_1, Q_2$ and $\lambda$. By \eqref{Eq8}  equation \eqref{Eq10} reduces to
   \begin{equation}\label{Eqss1}
   -\frac{MQ_2}{1-Q_2^2}=(1-\alpha)\frac{Q_1^2}{1-Q_1^2}\cdot
   \end{equation}
  Now $M$ can be expressed as
  \begin{eqnarray*}
   M&=& \lambda B_1+Q_1 A_1\\
   &=&(\lambda-1)B_1+B_1+Q_1 A_1\\
   &=&(1-\lambda)\frac{2Q_1}{1-Q_1^2}-Q_1.
   \end{eqnarray*}
   Thus equation\eqref{Eqss1} becomes, if $Q_1\neq0,$
    \begin{eqnarray*}
   Q_1 Q_2  \big(Q_1+(\alpha-1)Q_2\big)&=& (2\lambda-1)Q_2+(\alpha-1)
   Q_1,
   \end{eqnarray*}
 which is equivalent to
   $$
  \big((1-\alpha)+Q_1Q_2\big)\big(Q_1+(\alpha-1) Q_2\big)=\big(2\lambda-1-(1-\alpha)^2\big) Q_2.
   $$
   Combining this equation with \eqref{eq3} we get the system
   \begin{equation}\label{ss12}
 \left\{
\begin{array}{ll}
  \big((1-\alpha)+Q_1Q_2\big)\big(Q_1+(\alpha-1) Q_2\big)=\big(2\lambda-1-(1-\alpha)^2\big) Q_2 \\
 (1+Q_2^2)\big(Q_1+(\alpha-1) Q_2\big)=2(\lambda+\alpha-1) Q_2.
\end{array} \right.
     \end{equation}
     To solve this system  we distinguish  two cases.
     \vspace{0,2cm}

     {\bf Case $1$}: $Q_1+(\alpha-1)Q_2=0$. Since the ellipse $\Gamma_2$ is not a circle then $Q_2\neq0$
      and the second equation of the preceding system gives necessary $\lambda=1-\alpha.$
       Inserting this value into the first equation of \eqref{ss12} yields $\alpha=0$ and so $Q_1=Q_2$. The latter condition
        is impossible because the ellipses are confocal and different.
       \vspace{0,2cm}

       {{\bf Case $2$}: $Q_1+(\alpha-1)Q_2\neq0$}. Dividing the first equation in  \eqref{ss12}
       by the second we get
  $$
  \frac{1-\alpha+Q_1Q_2}{1+Q_2^2}=\frac{2\lambda-1-(1-\alpha)^2}{2(\lambda+\alpha-1)}\triangleq C.
  $$
   Hence
   $$
   1-\alpha+Q_1Q_2=C(1+Q_2^2).
   $$
     Multiplying the second equation of \eqref{ss12} by $Q_2$ and using the previous identity
     we see that
     $$
     \big(C+\alpha-1\big)\big(1+Q_2^2\big)^2=2(\lambda+\alpha-1)Q_2^2.
     $$
     Thus
     \begin{equation}\label{ident11}
     \frac{Q_2^2}{(1+Q_2^2)^2}=\frac{C+\alpha-1}{2(\lambda+\alpha-1)}\cdot
     \end{equation}
    Recalling that $1-\lambda=2\Omega$ elementary arithmetics leads to
    \begin{eqnarray*}
    \frac{C+\alpha-1}{2(\lambda+\alpha-1)}&=&\frac{\alpha^2+2\alpha(\lambda-1)}{4(\lambda+\alpha-1)^2}\\
    &=&\frac{\alpha^2-4\alpha\Omega}{4(\lambda-1+\alpha)^2}\cdot  \end{eqnarray*}

    Set $\rho= \frac{4Q_2^2}{(1+Q_2^2)^2}.$  Then
    equation \eqref{ident11} is
    $$
    4\rho\,\Omega^2+4\alpha(1-\rho)\Omega+\alpha^2(\rho-1)=0.    $$
    The solutions of the quadratic equation above are
    $$
    \Omega_{\pm}=\alpha\frac{(\rho-1)\pm\sqrt{1-\rho}}{2\rho},
    $$
   which can be readily written as
    $$
    \Omega_{+}=\alpha\,\frac{1-Q_2^2}{4}, \quad  \Omega_{-}=\alpha\frac{Q_2^2-1}{4 Q_2^2}.
    $$
    From the second equation in \eqref{ss12} the $Q_1$ associated to $\Omega_+$ is given by
    \begin{eqnarray*}
    Q_1&=&{Q_2}\left(\frac{2\alpha-4\Omega_+}{1+Q_2^2}+1-\alpha  \right)\\
    &=& Q_2.
    \end{eqnarray*}
    Since the ellipses are confocal this means that they are the same, which is not the case. The value
     of $Q_1$ associated to $\Omega_{-}$ is given by
    $$
 Q_1=Q_2\left(\frac{\alpha}{Q_2^2}+1-\alpha \right).
    $$
   Recall that the ellipses are confocal and $\overline{D_2}\subset D_1.$ Then $ 0 < Q_1 /Q_2 < 1,$  which is equivalent to
       $$
    -\frac{Q_2^2}{1-Q_2^2}<\alpha<0.
    $$
    In conclusion, the ellipses rotate with the same angular velocity $\Omega$ if and only if we have the relations
    \begin{equation*}\label{assump}
     \Omega=\alpha\,\frac{Q_2^2-1}{4 Q_2^2},\quad Q_1=Q_2\left(\frac{\alpha}{Q_2^2}+1-\alpha  \right)\quad
\hbox{and}\quad -\frac{Q_2^2}{1-Q_2^2}<\alpha<0,
    \end{equation*}
which completes the proof of Theorem \ref{impresul}.

\begin{gracies}
The authors are grateful to Joan Josep Carmona for an illuminating
conversation concerning the proof of Theorem \ref{impresul} and to
Luis Vega for suggesting that the study of doubly connected V-states
might be interesting. J. Mateu and J. Verdera acknowledge generous
support from the grants 2009SGR420 (Generalitat de Catalunya) and
MTM2010-15657 (Ministerio de Ciencia e Innovaci\'{o}n).
\end{gracies}

\end{document}